\documentclass[12pt,reqno]{amsart}
\usepackage[margin=1in]{geometry}
\usepackage{amsmath,amssymb,amsthm}
\usepackage{mathrsfs}
\usepackage{xcolor}
\usepackage{hyperref}
\allowdisplaybreaks
\usepackage{quotes}

\numberwithin{equation}{section}
\newtheorem{Thm}{Theorem}[section]
\newtheorem{Def}{Definition}[section]
\newtheorem{Hyp}{Hypothesis}[section]
\newtheorem{Lem}{Lemma}[section]
\newtheorem{Pro}{Proposition}[section]

\newtheorem{Rem}{Remark}[section]

\newcommand{\bm}{{\boldsymbol{m}}}
\newcommand{\bn}{{\boldsymbol{n}}}

\def\mR{\mathbb{R}}
 
\def\mN{\mathbb{N}}

\def\mC{\mathbb{C}}
\def\C{\mathrm{C}}

\def\R{\mathrm{R}}

\def\B{{\mathrm B}}

\def\t{\tau_N}

\def\i0t{\int_0^t}

\let\originalleft\left
\let\originalright\right
\renewcommand{\left}{\mathopen{}\mathclose\bgroup\originalleft}
\renewcommand{\right}{\aftergroup\egroup\originalright}

\begin{document}
	
\newcommand{\Addresses}{{
		\bigskip
		\footnote{
	\noindent  \textsuperscript{1} U. S. Air Force Research Laboratory, Wright Patterson Air Force Base, Ohio 45433, U. S. A.
	
	\par\nopagebreak
	\noindent  \textsuperscript{2} NRC-Senior Research Fellow, National Academies of Science, Engineering and Medicine,  U. S. Air Force Research Laboratory, Wright Patterson Air Force Base, Ohio 45433, U. S. A.		
		\par\nopagebreak \noindent
		\textit{e-mail:} \texttt{provostsritharan@gmail.com} 	$^*$Corresponding author.
	\par\nopagebreak
\noindent  \textit{e-mail:} \texttt{saba.mudaliar@us.af.mil}

}}}

\title[Filtering of Stochastic Nonlinear Wave Equations]{Filtering of Stochastic Nonlinear Wave Equations  \Addresses}
\author[ S. S. Sritharan and Saba Mudaliar]
{  Sivaguru S. Sritharan\textsuperscript{2*} and Saba Mudaliar\textsuperscript{1}}

\maketitle

\begin{abstract}
In this paper we will develop linear and nonlinear filtering methods for a large class of nonlinear wave equations that arise in applications such as quantum dynamics and laser generation and propagation in a unified framework. We consider both stochastic calculus and white noise filtering methods and derive measure-valued evolution equations for the nonlinear filter and prove existence and uniqueness theorems for the solutions. We will also study first order approximations to these measure-valued evolutions by linearizing the wave equations and characterize the filter dynamics in terms of infinite dimensional operator Riccati equations and establish solvability theorems. 
\end{abstract}



\section{Introduction}
Sensing and estimation of random fields has many applications in engineering sciences. One of the most powerful techniques developed in this context is filtering theory and has its origins in the works of Wiener and Kolmogorov \cite{Wiener1949, Kailath1974}. Filtering of Markov process was initiated by Stratenovich \cite{Strato1960}. However, the filtering of linear systems developed by Kalman has gained wide spread utilization due to its readily implementable form of tracking the state estimator and error covariance in a recurssive manner. Nonlinear filtering theory of Stratenovich  was further crystelized by several authors (see \cite{Kallianpur1980, Kallianpur1988} for a comprehensive discussion) leading to the famous Fujisaki-Kallianpur-Kunita and Zakai equations for the evolution of unnormalized as well as normalized conditional measures. Filtering of linear infinite dimensional systems was initiated by Bensoussan \cite{Bensoussan1971} to derive infinite dimensional Kalman filters. Filtering of nonlinear system of partial differential equations was initiated for the stochastic Navier-Stokes equations \cite{Sritharan1995} and stochastic reaction diffusion equations \cite{Hobbs1996}. This subject was further developed to incorporate L\'evy noise forced Navier-Stokes equations \cite{Fernando2013} and also classical and quantum spin systems \cite{Sritharan2023a}. In this paper we will further develop linear and nonlinear filtering theory to address a large class nonlinear wave equations with Gaussian and L\'evy forcings with several specific applications in laser generation and propagation and quantum dynamics. Many important applications of the above type can be captured in a unified mathematical model of nonlinear wave equations as observed by several authors \cite{Reed1975a, Reed1976,Reed1980, Strauss1978, Strauss1989}. We will develop our filtering study to such general class of nonlinear evolutions. 

We discuss next a summary of the sections of this paper. In the remaining part of this section we provide a general mathematical framework of semilinear wave equations in a framework suitable for semigroup treatment of this system, discuss the mathematical properties of the operators involved and also provide a solvability theorem for the deterministic nonlinear wave problem. In Section 2 we discuss four important applications in laser generation and propagation and demonstrate that all of these models can be cast in our unified framework. In Section 3 we discuss the solvability theory of stochastic nonlinear wave equations by casting the problem in the mild solution form. In section 4 we develop nonlinear filtering for the stochastic wave equation in the Ito calculus framework and in Section 5 we develop white noise filtering method. In Section 6 we consider linear stochastic wave equation and show that the nonlinear filtering equations of Ito calculus and white noise type result in infinite dimensional Kalman filters. We describe the time evolution of least square best estimate and the associated dynamics of error covariance in terms of infinite dimensional Riccati equations and establish their solvability. Finally we also wish to point out that to our knowledge this is the first time linear and nonlinear filtering method has been developed for stochastic wave equations.

We will consider a class of stochastic semilinear evolution of the type:

\begin{equation}
	\frac{\partial \varphi}{\partial t} = \left(- iA +  V(t)\right )\varphi +J(\varphi), \label{eqn:1.1}
\end{equation}
where $A$ is a self-adjoint unbounded operator, $V$ is a space-time random field, and $J(\cdot)$ is a nonlinear interaction function.
We will formulate the mild solution form of the above semilinear stochastic evolution  with the random potential $V$ represented by the generalized derivative of a Hilbert space valued stochastic process of Wiener, Poisson and L\'evy types.

  Let $(\Omega, \Sigma, \Sigma_{t},m)$ be a complete filtered probability space. It\^o calculus formulation is as follows:
\begin{equation}
	\varphi(t)= e^{-i A t}\varphi(0) +\int_{0}^{t}e^{-i A(t-s)}	 J(\varphi(s)) ds+\int_{0}^{t}e^{-i A(t-s)}\varphi(s) dM(s), \label{eqn:1.2}
\end{equation}

The random potential $V$ is modelled as the generalized time derivative of $M$, i.e. $V = dM/dt $ in the generalized sense and $M(t)$ is a local semi-martingale. Typical examples are:
\begin{enumerate}
	\item $M(t)=W(t)$ an $H$-valued Wiener process with covariance $Q$,
	\item  $dM(t) =  dW(t) +\int_{Z}\Psi (z)d\bar{\mathcal N}(t,z)$ where $\bar{\mathcal N}(\cdot,\cdot)$ is a (compensated) Poisson random (martingale) measure and $M(t)$ is an $H$-valued L\'{e}vy process. Here $H$ is a separable Hilbert space.
\end{enumerate}

We recall here the definition of L\'evy process \cite{Applebaum2009}:
\begin{Def} Let $Z$ be a Banach space. An $Z$-valued stochastic process $L(t)$ is a {\it L\'evy process} if
	\begin{enumerate} 
			\item $L(0)=0,$  a.s.;
			\item $L$ has stationary and independent increments;
			\item $L$ is stochastically continuous: for all bounded and measurable functions $\phi$ the function $t\rightarrow E[\phi(L(t))]$ is continuous on $\R^{+}$;
			\item $L$ has a.s. c\'adl\'ag paths. 
	\end{enumerate}
	
\end{Def}
{\bf Special Cases}: When the increments are Gaussian distributed the process is called Brownian motion (or Wiener process) and when the increments are Poisson distributed then it is called Poisson process.

L\'evy Process and Poisson Random Measure: Let $L$ be a real valued L\'evy process and let $A\in {\mathcal B}(\mR).$ We define the counting measure:
		\begin{displaymath}
			{\mathcal N}(t,A)=\# \left \{ s \in (0,T];\Delta L(s)=L(s)-L(s^{-})\in A\right \} \in \mN \cup \{\infty \}.
		\end{displaymath}
	then we can show that 
	\begin{enumerate}
			\item ${\mathcal N}(t,A) $ is a random variable over $(\Omega, {\mathcal F}, m)$;
			\item ${\mathcal N}(t,A) \sim$ Poisson$(t\nu(A))$ and ${\mathcal N}(t,\emptyset)=0;$
			\item For any disjoint sets $A_{1},\cdots, A_{n}$ the random variables ${\mathcal N}(t, A_{1}),\cdots, {\mathcal N}(t, A_{n})$ are pairwise independent.
	\end{enumerate}	
	Here $\nu(A)={\mathcal E}[{\mathcal N}(1, A)]$ is a Borel measure called L\'evy measure. {\it Compensated Poisson random measure} will then be the martingale:
	\begin{displaymath}
		\bar{{\mathcal N}} (t,A)={\mathcal N}(t,A)-t\nu(A), \hspace{.1in} \forall A\in {\mathcal B}(\mR).
	\end{displaymath}

Let $W(t)$ be an $H$-valued Wiener process with symmetric trace-class covariance operator $Q\in {\mathcal L}(H,H)$:
\begin{displaymath}
	E[(W(t),\psi)(W(\tau),\phi)]=t\wedge \tau (Q\psi,\phi),\hspace{.1in} \forall \phi,\psi\in H,
\end{displaymath}
and $\mbox{ Tr}Q <+\infty$ (finite trace). Using the eigensystem $\{ \lambda_{i}, e_{i}\}_{i=1}^{\infty}, e_{i}\in H, i\geq 1$ of $Q$ we can also express the infinite dimensional Wiener process $W(t)$ as a series expansion \cite{DaPrato2014}:
\begin{displaymath}
	W(t)=\sum_{i=1}^{\infty}\sqrt{\lambda_{i}}e_{i}\beta_{i}(t) \in H, \hspace{.1in} t\geq 0,
\end{displaymath}
where $\beta_{i}$ are standard independent Brownian motions and $\mbox{ Tr}Q=\sum_{i=1}^{\infty} \lambda_{i} <+\infty.$

For the pure jump case we can construct \cite{Curtain1975} $M(t)$ as an $H$-valued real compensated compound Poisson process $\bar{q}(t)\in H$:
\begin{displaymath}
	M(t)= \bar{q}(t)=\sum_{i=0}^{\infty} \bar{q}(t)e_{i}, 
\end{displaymath}
where $\{e_{i}\}_{i=0}^{\infty}\subseteq H$ is a basis in $H$. Here 
\begin{displaymath}
	E \{q_{i}(t)-q_{i}(s)\}=\mu_{i}(t-s),
\end{displaymath}
\begin{displaymath}
	E \left \{ (\bar{q}_{i}(t_{1})-\bar{q}_{i}(s_{1})) (\bar{q}_{j}(t_{2})-\bar{q}_{j}(s_{2}))  \right\}=0, \hspace{.1in} 0\leq s_{1} < t_{1}\leq s_{2} < t_{2}\leq T,
\end{displaymath}
\begin{displaymath}
E \left \{ (\bar{q}_{i}(t)-\bar{q}_{i}(s)) (\bar{q}_{j}(t)-\bar{q}_{j}(s))  \right\}=\lambda_{ij}(t-s), \hspace{.1in} 0\leq s\leq t\leq T
\end{displaymath}
where the compensated Poisson process $\bar{q}_{i}(t)=q_{i}(t)-\mu_{i} t $  and $ \sum_{i=0}^{\infty} \mu_{i}<\infty.	$ We will denote the covariance matrix as $\Lambda =\{\lambda_{ij}\}$:
\begin{displaymath}
\mbox{ Cov}(\bar{q}(t)-\bar{q}(s))	=\Lambda (t-s).
\end{displaymath}

Next we will describe the mathematical structure of operators $A$ and $J$ and will indicate that these properties are satisfied by the specific nonlinear wave models considered in the next section. We will also prove a solvability theorem for a deterministic semilinear evolution equation associated with (\ref{eqn:1.1}). 
\subsection{The Free Propagator and its $m$-Dissipative Generator}	
Let us begin with an important class of operators that cover the linear part of the equation \ref{eqn:1.1}. Although the class of nonlinear dissipative operators (or accretive operators) has led to powerful mathematical developments in nonlinear evolution equations \cite{Barbu1976}, we will only consider the linear case here. A linear operator $A$ with domain $D(A)$ in a Banach space $X$ is called dissipative \cite{Hille1996} if 
\begin{displaymath}
	\|u-\lambda Au\|\geq \|u\|, \forall u\in D(A)\mbox{ and } \forall \lambda >0.	
\end{displaymath}
Let $X$ be a Hilbert space. It can be shown that a linear operator $A$ in $X$ is dissipative if and only if
\begin{displaymath}
	\langle Au, u\rangle \leq 0, \hspace{.1in} \forall u\in D(A).
\end{displaymath}	
\begin{Def}
	An operator $A$ in $X$ is $m$-dissipative if
	
	(i)  $A$ is dissipative,
	
	(ii) $\forall \lambda >0$ and $\forall f\in X$ there exists $u\in D(A)$ such that $u-\lambda Au=f$.
\end{Def}
We will collect some relevant facts on $m$-dissipative operators on a Hilbert space in the theorem below (see also \cite{Cazenave1998}):	
\begin{Thm} \label{Thm:1.1}
	\begin{enumerate}
		\item If $A$ is $m$-dissipative in $X$ then $D(A)$ is dense in $X$.
		\item	If $A$ is a skew-adjoint operator in $X$ then $A$ and $-A$ are $m$-dissipative. 
		\item	If $A$ is a self-adjoint operator and $A\leq 0$, i.e., $\langle Au, u\rangle \leq 0, \forall u\in D(A)$ then $A$ is $m-dissipative$.
	\end{enumerate}
\end{Thm}

We recall here that if $A$ is self-adjoint then $iA$ is skew-adjoint:
\begin{displaymath}	
	(iA)^{*}=-iA^{*}=-iA.
\end{displaymath}		

\begin{Rem}
	The following examples of skew-ajoint operators are encountered later in the next section.		
	
	(1) Free Schrodinger operator:
	\begin{displaymath}
		A=i\Delta.
	\end{displaymath}
	Here $\Delta$ can be defined as a self-adjoint operator with a suitable choice of domain and hence $i\Delta$ is a skew-adjoint operator and is $m$-dissipative.
	
	(2) Wave operator: let $B=\sqrt{ -\Delta +k_{0}^{2}I}$, $D(B)=\left \{u \in L^{2}(\mR^{3}); Bu\in L^{2}(\mR^{3})\right \}$, and we define 
	\begin{displaymath}
		A= i\begin{bmatrix}
			0 & I\\
			-B^{2} & 0
		\end{bmatrix}.
	\end{displaymath}
	Then $A$ is a self-adjoint operator and $iA$ is skew-adjoint and $m$-dissipative.
	
	Moreover, the propagator associated with $iA$ is given by \cite{Segal1963,Reed1976}:
	\begin{displaymath}
		e^{-itA}=\begin{bmatrix}
			cos (tB) & B^{-1}sin (tB)\\
			-B sin(tB) & cos (tB)
		\end{bmatrix}.	
	\end{displaymath}
	
	(3) Dirac operator
	\begin{displaymath}
		D_{e}=-i \alpha\cdot\nabla + m\beta.	
	\end{displaymath}
	To understand this operator we describe it component-wise by setting $D_{e}u=v \in \mC^{4}$. Here the four dimensional complex vector space $\mC^{4}$ is called spinor space.
	\begin{displaymath}
		v_{j}=(D_{e}u)_{j}=i^{-1} \sum_{l=1}^{3}\sum_{k=1}^{4}(\alpha_{l})_{jk}\frac{\partial u_{k}}{\partial x_{l}}(x) +\sum_{k=1}^{4}\beta_{jk}u_{k}(x),
	\end{displaymath}
	where $\alpha_{k}$ and $\beta$ are Hermitian matrices satisfying the commutation relations
	\begin{displaymath}
		\alpha_{j}\alpha_{k}+\alpha_{k}\alpha_{j}=2\delta_{jk}I, \hspace{.1in} j, k=1,2,3,4.
	\end{displaymath}
	Here $\alpha_{4}=\beta$, and $I$ is the $4\times 4$ unit matrix.
	
	$D_{e}$ is a self-adjoint operator \cite{Kato1976} and hence $iD_{e}$ is a skew-adjoint operator and is $m$-dissipative.

	(4) Maxwell-Dirac operator
	\begin{displaymath}
		A= \begin{bmatrix}
			D_{e} & 0\\
			0 & A_{W}
		\end{bmatrix}
	\end{displaymath}	
	with $	D_{e}$ the Dirac operator and $A_{W}$ the wave operator defined above. Then $A$ is a self-adjoint operator and $iA$ is skew-adjoint and $m$-dissipative.
	
\end{Rem}

Let us now recall three well-known theorems\cite{Hille1996} that are relevant in characterizing the free propagator semigroup (actually group) generated by the linear part $iA$ and its bounded perturbations that we encounter in this paper.

\begin{Thm}\cite{Yosida1991}\label{Thm:1.2} (Hille-Yosida-Phillips) A linear operator $A$ is the generator of a contraction semigroup $S(t)$ in X if and only if $A$ is $m$-dissipative with a dense domain $D(A)$.
\end{Thm}
If $A$ is a skew-adjoint operator then $S(t)$ can be extendable to a group that is unitary.

\begin{Thm}\label{Thm:1.3} \cite{Kato1976} (Stone's Theorem)  Let ${\mathcal N}$ be a von Neumann algebra and let $\{U_{t}\}_{t\in \R} \subset {\mathcal N}$ be a group of unitary operators that is strongly continuous. Then there is a unique self-adjoint operator $A$ affiliated to ${\mathcal N}$ (i.e., $(A+iI)^{-1}\in {\mathcal N}$), the Stone generator, such that $U_{t}= e^{itA}$.
\end{Thm}	
The following theorem is suitable for the linearized wave equation (linearization of \ref{eqn:1.1} about a constant state $X_{0}$) giving a perturbed generator of the form:  ${\mathcal A}=-iA +J'(X_{0})$.
\begin{Thm}\cite{Pazy1983} \label{Thm:1.4}
	Let $X$ be a Banach space and let $A$ be the infinitesimal generator of a $C_{0}$ semigroup denoted by $e^{At}$ on $X$ satisfying $\|S(t)\|\leq Me^{\omega t}$. If $B$ is a bounded linear operator on $X$ 	then ${\mathcal A}=A+B$ is the infinitesimal generator of a $C_{0}$ semigroup denoted by $e^{(A+B)t}$ on $X$ satisfying $\|e^{(A+B)t}\|\leq Me^{(\omega +M\|B\|) t}$.
\end{Thm}
We will consider time dependent perturbations of the form ${\mathcal A}(t)=-iA +J'(X_{0}(t))$ and associated generation theorems later in Section 6.

\subsection{Continuity Properties of the Nonlinearity}	
Let us define a class of nonlinearities and their continuity properties that would enable the solvability theorems of deterministic and stochastic type. Leter we will demonstrate that this class covers the laser generation and interaction models considered in this paper.
\begin{Hyp} \label{Hyp:1.1}
	Let $A$ be a self-adjoint operator on a Hilbert space $H$ and $N$ be a positive integer. Suppose that $J$ is a densely defined nonlinear mapping on $H$ so that $J:D(A^{j})\rightarrow D(A^{j})$ for $0 \leq j\leq N$,:	
	\begin{enumerate}
		\item $\|A^{j}J(\varphi)\|\leq C \left (\|\varphi\|,\cdots, \|A^{j}\varphi\|\right) \|A^{j}\varphi\|,$
		\item $\|A^{j}(J(\varphi)-J(\psi))\|\leq C \left (\|\varphi\|, \|\psi\|,\cdots, \|A^{j}\varphi\|,\|A^{j}\psi\|\right) \|A^{j}(\varphi-\psi)\|,$
	\end{enumerate}
	for $j=0,1,\cdots,N$ and all $\phi, \psi\in D(A^{N})$, with constants $C(\cdot)$ monotonically increasing functions of the norms indicated.
\end{Hyp}	
\subsection{Local and Global Solvability Theory for the Unified Abstract Deterministic Semilinear Equation}
In this section we will recall \cite{Reed1975a, Reed1976} the solvability results for the deterministic abstract semilinear evolution in the mild form. These theorems can serve as guidelines for the stochastic problem. We consider:
\begin{equation}
	\varphi(t)= e^{i A t}\varphi(0) +i\int_{0}^{t}e^{-i A(t-s)}	 J(\varphi(s)) ds, \label{eqn:2.1}
\end{equation}
which corresponds to the strong form
\begin{equation}
	\frac{\partial \varphi}{\partial t} = - iA \varphi +J(\varphi), \label{eqn:2.2}
\end{equation}
\begin{equation}
	\varphi(0)=\varphi_{0}. \label{eqn:2.3}
\end{equation}

\begin{Thm} \label{Thm:1.5} \cite{Reed1975a, Reed1976}
	Let $A$ be a self-adjoint operator on a Hilbert space $H$ with domain $D(A)\subset H$ and let $J:D(A^{j})\rightarrow D(A^{j})$ for $0\leq j\leq N$, be a nonlinear mapping which satisfies Hypothesis \ref{Hyp:1.1}. Then, given $\varphi_{0}\in D(A^{N})$, there is a $T>0$ and a unique $D(A^{N})$-valued function $\varphi(t)$ on $[0,T)$ which satisfies (\ref{eqn:2.2}), (\ref{eqn:2.3}). For each set of the form $\left \{\varphi\vert \|A^{j}\varphi\|\leq a_{j}, j=0, \cdots, N\right \}$, $T$ can be chosen uniformly for all $\varphi_{0}$ in the set. The solution satisfies the estimate:
	\begin{equation}
		\sup_{t\in [0,T)} \sum_{j=0}^{N}\|A^{j}\varphi(t)\| \leq \sum_{j=0}^{N}	\|A^{j}\varphi_{0}\|e^{C(T)}.
	\end{equation}
 Suppose that in addition to the Hypothesis \ref{Hyp:1.1}, $J$ satisfies the following property: If $\varphi(t)$ is a $j$ times strongly differentiable $D(A^{N})$-valued function so that $\varphi^{(k)}(t)\in D(A^{N-k})$ and $A^{N-k}\varphi^{(k)}(t)$ is continuous for all $k\leq j$, then $J(\varphi(t))$ is $j$ times continuously differentiable, $(\frac{d}{dt})^{j}J(\varphi(t))\in D(A^{N-j-1})$ and $A^{N-j-1}(\frac{d}{dt})^{j}J(\varphi(t))$ is continuous. Then the solution is $N$ times strongly continuously differentiable and $\varphi^{(j)}(t)\in D(A^{N-j})$ for all $j\leq N$.

 If instead of the estimate (1) of Hypothesis  \ref{Hyp:1.1} we have the estimate for the nonlinearity:
	\begin{displaymath}
		\|A^{j}J(\varphi)\|\leq C\left (\|\varphi\|, \cdots, \|A^{j-1}\varphi\|\right ) \|A^{j}\varphi\|,	\hspace{.1in} 1\leq j\leq N.
	\end{displaymath}
	Suppose that on any finite interval $[0,T)$ where the strong solution $\varphi(t)$ of equation \ref{eqn:2.1} exists, $\|\varphi(t)\|$ is bounded. Then the strong solution exists globally in $t$ and $\varphi(t)\in D(A^{N})$ for all $t$.

	Let $J$ satisfy the Hypothesis \ref{Hyp:1.1} for all $N$ so that any strong solution of \ref{eqn:2.1} is a-priori bounded on any finite interval. If $\varphi_{0}\in \bigcap_{j=1}^{\infty}D(A^{j})$, then the solution $\varphi(t)$ of \ref{eqn:2.1} is infinitely strongly differentiable in $t$ and each time derivative has values in $\bigcap_{j=1}^{\infty}D(A^{j})$.
\end{Thm}
\section{Laser Propagation and Generation Models}
In this section we will give formal derivations of various Laser propagation and generation models and discuss the individual mathematical characteristics of each of the models.
\subsection{Continuous and Pulse Wave Models: The Stochastic Paraxial and Klein-Gordon Equations}

The paraxial equation is a very well-known model widely used in the literature on the acoustic wave and electromagnetic wave propagation in random media \cite{Tatarski1961, Papa1973, Tap1977, Strohbehn1978, Sp2002,Sprangle2003, Gustafsson2019, Sritharan2023}. As in these papers, we derive from the Maxwell's equations a wave equation for the electric field $E(x,t)$:
\begin{equation}
	\Delta E(x,t) - \frac{1}{c^{2}} \frac{\partial^{2}}{\partial t^{2}}	(n^{2}(x,t) E(x,t))= -2\nabla (E(x,t)\cdot\nabla \log (n(x,t))), \label{eqn2.1}	
\end{equation}
where $n(x,t)$ is the refractive index of the medium and $c$ is the speed of light. We assume that the time scale of fluctuations in the medium is much slower than the light speed and invoke further simplifications based on this assumption. Thus neglecting the right hand side and also the $n^{2}$ term out of the time derivative we arrive at
\begin{equation}
	\Delta E(x,t) - \frac{n^{2}(x,t)}{c^{2}} \frac{\partial^{2}}{\partial t^{2}}	 E(x,t)=0. \label{eqn2.2}
\end{equation}
Substituting a plane wave solution $E(x_{1},x_{2},x_{3},t)=\psi(x_{1},x_{2},x_{3})\exp(ik x_{3}-i\omega t)$ and neglecting the back-scatter term  $\frac{\partial^{2}\psi(x_{1},x_{2},x_{3})}{\partial x_{3}^{2}}$ (using simple scaling argument, see for example \cite{Papa1973, Strohbehn1978}) we arrive at the nonlinear random paraxial equation:
\begin{displaymath}
	i \frac{\partial \psi(x_{1},x_{2},x_{3})}{\partial x_{3}}+ \Delta_{\bot}\psi(x_{1},x_{2},x_{3}) +\chi(x_{1},x_{2},x_{3})\psi(x_{1},x_{2},x_{3})
\end{displaymath}
\begin{equation}
	\pm
	\vert \psi(x_{1},x_{2},x_{3})\vert^{p-1}\psi(x_{1},x_{2},x_{3})=0. \label{eqn2.3}
\end{equation}
where $\Delta_{\bot} $ is the two dimensional Laplacian in the variables $x_{2},x_{3}$ and $\chi$ is a random field that depends on the medium and the field strength. 
Renaming the time-like variable $x_{3}$ as $t$ and suppressing the actual time variable $t$ we end up with the {\em two dimensional} linear or nonlinear stochastic Schr\"odinger equation:
\begin{displaymath}
	i \frac{\partial \psi(x_{1},x_{2},t)}{\partial t}+ \Delta_{\bot}\psi(x_{1},x_{2},t) +\chi(x_{1},x_{2},t)\psi(x_{1},x_{2},t)
\end{displaymath}
\begin{equation}
	\pm\vert \psi(x_{1},x_{2},t)\vert^{p-1}\psi(x_{1},x_{2},t)=0. \label{eqn2.4}
\end{equation}
On the other hand, using a wave envelop type representation 
\begin{displaymath}
	E(x_{1},x_{2},x_{3},t)=\psi(x_{1},x_{2},x_{3},t)\exp(ik x_{3}-i\omega t)
\end{displaymath}
in \ref{eqn2.1},we arrive at a time dependent random nonlinear wave equation \cite{Sprangle2003}:
\begin{equation}
	\frac{\partial^{2}}{\partial t^{2}}\psi(x,t)	-(\Delta-k_{0}^{2})\psi(x,t)+\chi(x,t)\psi(x,t)\pm\vert\psi (x,t)\vert^{p-1}\psi(x,t)=0,\label{eqn2.5}
\end{equation}
where $\chi$ is a random field that depend on the medium. Here we have also neglected certain first order derivative terms that appear in \cite{Sprangle2003} to arrive at nonlinear random Klein-Gordon type wave equation. Setting $v=\partial_{t}\psi$ we reframe the above dynamics as:
\begin{displaymath}
	\frac{\partial}{\partial t} \begin{bmatrix}
		\psi\\
		v
	\end{bmatrix}
	- \begin{bmatrix}
		0 & I\\
		-B^{2} & 0
	\end{bmatrix}
	\begin{bmatrix}
		\psi\\
		v
	\end{bmatrix}
	+
	\begin{bmatrix}
		0\\
		\pm 	\vert \psi\vert^{p-1}\psi
	\end{bmatrix}
	+
	\begin{bmatrix}
		0\\
		\chi \psi
	\end{bmatrix}
	=0,
\end{displaymath}
where $B=\sqrt{ -\Delta +k_{0}^{2}I}$.

\begin{displaymath}
	H= D(B)\oplus L^{2}(\mR^{3}) \mbox{ and } D(A)= D(B^{2})\oplus D(B).
\end{displaymath}

Then we have the following estimate \cite{Reed1976}:

\begin{Lem}
	Setting $\varphi =\langle \psi, v\rangle$ and $J(\varphi)=\langle 0, \vert \psi\vert^{2}\psi\rangle$, for $\varphi_{1},\varphi_{2}\in H$, $J$ satisfies
	\begin{enumerate}
		\item $\|J(\varphi_{1})\|\leq K\|\varphi_{1}\|^{3}$,
		\item $\|J(\varphi_{1})-J(\varphi_{2})\|\leq C(\|\varphi_{1}\|,\|\varphi_{2}\|) \|\varphi_{1}-\varphi_{2}\|$,
		\item $\|AJ(\varphi_{1})\|\leq K\|\varphi_{1}\|^{2}\|A \varphi_{1}\|$,
		\item $\|A(J(\varphi_{1})-J(\varphi_{2}))\|\leq C(\|\varphi_{1}\|,\|\varphi_{2}\|,\|A \varphi_{1}\|,\|A \varphi_{2}\|) \|A\varphi_{1}-A\varphi_{2}\|$.
	\end{enumerate}
\end{Lem}

\subsection{ Maxwell-Dirac Equations}
Maxwell-Dirac equations describing classical quantum electrodynamics can also be regarded as the fundamental governing equations for free electron lasers \cite{Madey2010, Becker1979} and high power microwave generation dynamics. Rigorous theory of these models in the detrministic setting has a long history \cite{Dyson1949, Gross1966, Chadam1973, Flato1987, Strauss1978}. We will give a brief overview of this model and indicate how it fits in the unified mathematical theory and stochastic analysis developed in this paper. We start with the Maxwell's equations:
\begin{displaymath}
	\nabla \cdot E =\rho, \hspace{.1in} \nabla \cdot B=0,
\end{displaymath}
\begin{displaymath}
	\nabla \times E +\partial_{t}B=0, \hspace{.1in} \nabla \times B -\partial_{t}E =J.
\end{displaymath}
Dirac equation
\begin{displaymath}
	\left ( \alpha^{\mu}D_{\mu}+ m\beta\right)\psi=0.
\end{displaymath}
Here $m\geq 0$, $\alpha^{\mu}, \beta$ are $4\times 4 $ Dirac matrices. 

$E,B:\mR^{1+3}\rightarrow \mR$ are electric field and magnetic field respectively and $\psi:\mR^{1+3}\rightarrow \mC^{4}$ is the Dirac spinner field. 
Here $\alpha $ are linear operators in spin space that satisfy $\alpha^{\mu}\alpha^{\nu}+\alpha^{\nu}\alpha^{\mu}=2 g^{\mu \nu}$. $g^{00}=1, g^{11}=-1, g^{\mu\nu}=0$ for $\nu \not = \mu$ and $\alpha^{0*}=\alpha^{0}, \alpha^{1*}=-\alpha^{1}$.

We represent the electromagnetic fields by real four dimensional vector potential $A_{\mu}$, $\mu=0,1,2,3$:
\begin{displaymath}
	B=\nabla A, \hspace{.1in} E=\nabla A_{0}-\partial_{t}A, \mbox{ with  } A=(A_{1},A_{2}, A_{3}).
\end{displaymath}
The couplings are:
\begin{displaymath}
	J^{\mu}=\langle \alpha^{\mu}\psi,\psi\rangle_{\mC^{4}},
\end{displaymath}
\begin{displaymath}
	\rho=J^{0}=\vert \psi\vert^{2},\hspace{.1in} J=(J^{1}, J^{2}, J^{3}),
\end{displaymath}
the gauge covariant derivative
\begin{displaymath}
	D_{\mu}= \partial^{(A)}_{\mu}=\frac{1}{i}\partial_{\mu} -A_{\mu}.
\end{displaymath}
The above system is invariant under the transform:
\begin{displaymath}
	\psi\rightarrow \psi'=e^{i\chi}\psi, \hspace{.1in} A_{\mu}\rightarrow A'_{\mu}=A_{\mu}+\partial_{\mu}\chi,
\end{displaymath}
for any $\chi:\mR^{1+3}\rightarrow \mR$. 

We impose Lorenz gauge condition:
\begin{displaymath}
	\partial^{\mu}A_{\mu}=0    \mbox{  which is the same as   } \partial_{0}A_{0}=\nabla\cdot A.
\end{displaymath}
This results in the Maxwell-Dirac system:
\begin{equation}
	\square A_{\mu} =(\Delta -\partial_{0}^{2})A_{\mu} =  -\langle \alpha_{\mu}\psi, \psi\rangle_{\C^{4}}, \label{eqn:2.9}
\end{equation}

\begin{equation}
	\left ( -i \alpha^{\mu}\partial_{\mu} + m\beta \right )\psi =  A_{\mu}\alpha^{\mu} \psi, \label{eqn:2.10}
\end{equation}
along with the Lorenz gauge condition.
Now denoting the operator:
\begin{displaymath}
	A_{W}=- i\begin{bmatrix}
		0 & I\\
		-B_{0}^{2} & 0
	\end{bmatrix}
\end{displaymath}
where $B_{0}=\sqrt{-\Delta+k_{0}^{2}I}$.

The Dirac operator as $D_{e}=-i \alpha\cdot\nabla + m\beta $, 
\begin{displaymath}
	D_{e}^{2}=\left(
	\begin{array}{ccccc}
		B_{e}^{2}                                    \\
		& B_{e}^{2}              &   & \text{\huge0}\\
		&               & B_{e}^{2}                 \\
		& \text{\huge0} &   & B_{e}^{2}            \\
	\end{array}
	\right),
\end{displaymath}
where $B_{e}=\sqrt{-\Delta+m^{2}I}$.

Now setting $v=(A, \partial_{t}A)$ we can write the Maxwell-Dirac system as a semilinear evolution:
\begin{equation}
	\frac{\partial}{\partial t} \begin{bmatrix}
		\psi\\
		v
	\end{bmatrix}
	+ i\begin{bmatrix}
		D_{e} & 0\\
		0 & A_{W}
	\end{bmatrix}
	\begin{bmatrix}
		\psi\\
		v
	\end{bmatrix}
	+
	\begin{bmatrix}
		-	A_{\mu}\alpha^{\mu} \psi\\
		(0,i\langle \alpha_{\mu}\psi, \psi\rangle_{\mC^{4}} +k_{0}^{2})
	\end{bmatrix}
	=0. \label{eqn:2.11}
\end{equation}
We will take escalated energy spaces \cite{Chadam1972, Chadam1973, Reed1976} as state space:
\begin{displaymath}
	H= (\oplus_{i=0}^{3}D(B_{e}^{3}))\oplus (D(B_{0}^{3}) \oplus D(B_{0}^{2})).
\end{displaymath}
Then the operator:
\begin{displaymath}
	A=
	\begin{bmatrix}
		D_{e} & 0\\
		0 & A_{W}
	\end{bmatrix}	
\end{displaymath}
is self-adjoint on
\begin{displaymath}
	D(A)= (\oplus_{i=0}^{3}D(B_{e}^{4}))\oplus (D(B_{0}^{4}) \oplus D(B_{0}^{3})).
\end{displaymath}
The properties of the source term $J$ as required in the Hypothesis \ref{Hyp:1.1} and the main theorems are verified in \cite{Chadam1972, Chadam1973, Chadam1973b,Chadam1974, Reed1976}.

\subsection{Sine-Gordon Equation}
Sine-Gordon equation is a well-known model in quantum field theory and soliton theory. In \cite{Sritharan2023} we have initiated the stochastic quantization of this model and here we indicate that the hypotheses needed for the main theorems of this paper are satisfied.

\begin{displaymath}
	\frac{\partial^{2}}{\partial t^{2}}u-\Delta u +k_{0}^{2}u -g \mbox{ sin} (u) +V(t)u=0.
\end{displaymath}
Setting $v=\partial_{t}u$ and denoting $B=\sqrt{ -\Delta +k_{0}^{2}I}$ we reframe the above dynamics as:
\begin{displaymath}
	\frac{\partial}{\partial t} \begin{bmatrix}
		u\\
		v
	\end{bmatrix}
	- \begin{bmatrix}
		0 & I\\
		-B^{2} & 0
	\end{bmatrix}
	\begin{bmatrix}
		u\\
		v
	\end{bmatrix}
	+
	\begin{bmatrix}
		0\\
		g \mbox{ sin} (u)
	\end{bmatrix}
	+
	\begin{bmatrix}
		0\\
		V(t)u
	\end{bmatrix}
	=0,
\end{displaymath}
which is in the form \ref{eqn:1.1}. Let $H=D(B)\oplus L^{2}(\mR^{n})$, and $D(A)=D(B^{2})\oplus D(B)$. We then have \cite{Reed1976}:

\begin{Lem} Setting $ \varphi=\langle u, v\rangle$, $J(\varphi)=\langle 0, \mbox{ sin} (u)\rangle$ we have the following estimates: 
	\begin{enumerate}
		\item $\|J(\varphi)\|\leq \|\varphi\|,$
		\item $\|AJ(\varphi)\|^{2}\leq K \|\varphi\|^{2}$,
		\item $\|J(\varphi)-J(\psi)\|\leq K \|\varphi-\psi\|$,
		\item $ \|A(J(\varphi)-J(\psi))\|\leq K \|\varphi-\psi\|\|A\varphi\|+\|\varphi-\psi\|$.
	\end{enumerate}
\end{Lem}

\section{It\^o Calculus of the nonlinear stochastic wave equation}
In this section we will treat the mild solution form \ref{eqn:1.2} using It\^o calculus method combined with method of stochastic semilinear evolution in the spirit of \cite{DaPrato1996, DaPrato2014, Mohan2017}. The main theorem in this section is proven using techniques similar to the paper \cite{Mohan2017}.

Let us now define a stopping time $\tau_{\Lambda}$ as:
\begin{displaymath}
	\tau_{\Lambda}:=\inf\left \{ t>0; \sup_{0\leq j\leq N-1}\|A^{j}\varphi(t)\| >\Lambda \right\}.
\end{displaymath}

\begin{Thm} \label{Thm:3.1}
	Suppose $A$ and $J$ satisfy the Hypothesis \ref{Hyp:1.1}. Then there is a stopping time $\tau_{\Lambda}(\omega)>0$ such that there is a unique $\Sigma_{t}$-adapted $H$-valued stochastic process $\varphi(t,\omega), t\in (0,\tau_{\Lambda}), \omega \in \Omega$ with bound
	\begin{displaymath}
		E [	\sup_{t\in (0,\tau_{\Lambda})}\sum_{j=0}^{N}\|A^{j}\varphi(t)\|^{2}]\leq C(T, Q)E[\sum_{j=0}^{N} \|A^{j}\varphi_{0}\|^{2} ], 
	\end{displaymath}
	that satisfies 
	\begin{displaymath}
		\varphi(t)= e^{-i A t}\varphi_{0} +\int_{0}^{t}e^{-i A(t-s)}	 J(\varphi(s)) ds+\int_{0}^{t}e^{-i A(t-s)}\varphi(s) dW(s)
	\end{displaymath}
\begin{equation}
	+\int_{0}^{t}\int_{Z} e^{-i A(t-s)}\varphi(s) d\bar{\mathcal N}(s)ds, \hspace{.1in} 0\leq t\leq \tau_{\Lambda},
\end{equation}
	\begin{displaymath}
		\varphi(0)=\varphi_{0}\in H,
	\end{displaymath}
	with probability one in the time interval $(0,\tau_{\Lambda})$ and for a given $0 <\rho <1$ 
	\begin{displaymath}
		m\left \{\omega: \tau_{\Lambda}(\omega) >\rho \right\}\geq 1-\rho^{2}M,
	\end{displaymath}
	with $M$ depending on $\varphi_{0}, Q$ and independent of $\rho$. 
\end{Thm}

\begin{proof}
	For each $T>0$, let $X_{T}$ denote the set of continuous $D(A^{N})$-valued $\Sigma_{t}$-adapted stochastic processes $\varphi(\cdot)$ in the interval $[0,T)$ such that
	\begin{displaymath}
		\|\varphi(\cdot)\|_{X_{T}}^{2}:= E\left [\sup_{t\in [0,T)} \sum_{j=0}^{N}\|A^{j}\varphi(t)\|^{2}\right ]<\infty.
	\end{displaymath}
	For a fixed $\epsilon>0$, and given $\varphi_{0}\in D(A^{N})$ let
	\begin{displaymath}
		X_{T,\epsilon, \varphi_{0}}:=\left \{ \varphi(\cdot) \vert \varphi(0)=\varphi_{0},\hspace{.1in} \|\varphi(\cdot) -e^{-iA\cdot }\varphi_{0}\|_{X_{T}}\leq \epsilon \right \}.
	\end{displaymath}
	The goal would be to show that for small enough time $T$ the map ${\mathcal J}:X_{T,\epsilon, \varphi_{0}}\rightarrow X_{T,\epsilon, \varphi_{0}}$:
	\begin{equation}
		({\mathcal J}\varphi)(t)= e^{-i A t}\varphi_{0} +\int_{0}^{t}e^{-i A(t-s)}	 J(\varphi(s)) ds+\int_{0}^{t}e^{-i A(t-s)}\varphi(s) dW(s)		+\int_{0}^{t}\int_{Z} e^{-i A(t-s)}\varphi(s) d\bar{\mathcal N}(s)ds,
	\end{equation}
	is a contraction on $X_{T,\epsilon, \varphi_{0}}$.

	The following inequality is a consequence of the well-known stochastic convolution estimate \cite{DaPrato1996, DaPrato2014, vanNeerven2020}:
	\begin{displaymath}
		E\left [ \sup_{0\leq s\leq t} \vert| \int_{0}^{s}e^{-iA(t-r)}\varphi(r)dW(r)\vert|^{2}\right]\leq K E\left [\int_{0}^{t}\vert | \varphi(r)\sqrt{Q}\vert |_{HS}^{2}dr \right] 
	\end{displaymath}
	\begin{displaymath}
		\leq K (\mbox{tr}Q)E\left [\int_{0}^{t}\vert | \varphi(r)\vert |^{2}dr \right].
	\end{displaymath}
	In fact we have
	\begin{displaymath}
		E\left [ \sup_{0\leq s\leq t} \vert| \int_{0}^{s}A^{j} e^{-iA(t-r)}\varphi(r)dW(r)\vert|^{2}\right]
	\end{displaymath}
	\begin{displaymath}
		\leq K E\left [\int_{0}^{t}\vert | A^{j}\varphi(r)\sqrt{Q}\vert |_{HS}^{2}dr \right] 
	\end{displaymath}
	\begin{displaymath}
		\leq K (\mbox{tr}Q)E\left [\int_{0}^{t}\vert |A^{j} \varphi(r)\vert |^{2}dr \right]\leq K T (\mbox{tr}Q) E \left [\sup_{0\leq s\leq t}\vert |A^{j} \varphi(t)\vert^{2}\right ].
	\end{displaymath}
	Hence
	\begin{displaymath}
		\sum_{j=0}^{N}  E\left [ \sup_{0\leq s\leq t} \vert| \int_{0}^{s}A^{j} e^{-iA(t-r)}\varphi(r)dW(r)\vert|^{2}\right]  
	\end{displaymath}
	\begin{displaymath}
		\leq K T (\mbox{tr}Q)\sum_{j=0}^{N} E\left [\sup_{0\leq s\leq t}\vert |A^{j} \varphi(s)\|^{2}\right].	
	\end{displaymath}
The jump integral is estimated similarly \cite{Zhu2017, Mohan2017}
\begin{displaymath}
	E\left [ \sup_{0\leq t\leq \tau}\|\int_{0}^{t}\int_{Z}A^{j} e^{-i A(t-s)}\varphi(s) d\bar{\mathcal N}(s)ds\|^{2}\right ] \leq C E\left [ \int_{0}^{\t}\int_{Z}\|A^{j}\varphi(s)\|^{2}\nu(dz)ds\right ].
\end{displaymath}
	Now noting that due to Hypothesis \ref{Hyp:1.1} on $J$,
	\begin{displaymath}
		E\left [\|\int_{0}^{\tau_{\Lambda}}e^{-A(t-s)}A^{j}J(\varphi(s))ds\|^{2}\right ] \leq C(\Lambda,T ) E\left [\sup_{0\leq t\leq \tau_{N}}\| A^{j}\varphi(t)\|^{2}\right ],
	\end{displaymath}
	we can arrive at
	\begin{displaymath}
		E\left [\sup_{0\leq t\leq \tau_{\Lambda}}\sum_{j=0}^{N}	\|A^{j}({\mathcal J}\varphi)(t)-e^{-iAt}A^{j}\varphi_{0}\|^{2}\right ]
	\end{displaymath}
	\begin{displaymath}
		\leq C (\Lambda,\mbox{tr}Q) E\left [\sup_{t\in [0,\tau_{\Lambda})}\sum_{j=0}^{N}\|A^{j}\varphi(t)\|^{2}\right].	
	\end{displaymath}
	This completes the fixed point argument. The tail probability estimate can be obtained by arguments similar to \cite{Mohan2018}.
\end{proof}

We also recall here that as an $H$-valued square integrable random variable $\varphi(\cdot)\in L^{2}(\Omega, H)$ we have an infinite dimensional multiple Wiener expansion \cite{Alshanskiy2014}:
\begin{displaymath}
	\varphi(t)=\sum_{0}^{\infty}I_{n}(f_{n}),\hspace{.1in} f_{n}\in \hat{L}^{2}((0,\tau)^{n}:{\mathcal L}_{2}^{n}(H^{\otimes n};H)), n\in \mN,
\end{displaymath}
with
\begin{displaymath}
	E[\|\varphi\|]^{2} =\sum_{0}^{\infty}n! \|f_{n}\|^{2}_{L^{2}((0,\tau)^{n}:{\mathcal L}_{2}^{n}(H^{\otimes n};H))}.	
\end{displaymath}
where
\begin{displaymath}
	I_{n}(\zeta)=n! \int_{0}^{\tau}\int_{0}^{t_{n-1}}\cdots \int_{0}^{t_{1}}\zeta(t_{1},t_{2},\cdots, t_{n})dW(t_{1})\cdots dW(t_{n}),
\end{displaymath}
with $\zeta: [0,\tau]^{n}\times \Omega \rightarrow {\mathcal L}_{2}^{n}(H^{\otimes n};H).$ The multiple Wiener integrals satisfy the following orthogonality relations:
\begin{displaymath}
	E[(I_{n}(f), I_{m}(g))_{H}]=
	\left\{
	\begin{array}{ll}
		0  & \mbox{if } n\neq m \\
		(n!)^{2}	(f,g)_{L^{2}((0,\tau)^{n}:{\mathcal L}_{2}^{n}(H^{\otimes n};H))} & \mbox{if } m=n.
	\end{array}
	\right.
\end{displaymath} 

For the case of Levy noise a similar expansion in terms of Poisson random measures (along with Gaussian integrals above) can be developed as in \cite{Ito1956, Holden2010, Sritharan2023}.

\section{Ito calculus formulation of nonlinear filtering}
Nonlinear filtering theory for nonlinear stochastic partial differential equations of fluid mechanics and reacting and diffusing systems was initiated in \cite{Sritharan1995, Hobbs1996} and further developed for stochastic models of fluid mechanics with L\'evy noise in \cite{Sritharan2010, Sritharan2011, Fernando2013}. In \cite{Sritharan2023a} nonlinear filtering equations are derived for a class of classical and quantum spin systems. The nonlinear filtering problem for the class of nonlinear wave equations studied in this paper is formulated as follows. The key mathematical equations of  stochastic calculus method of nonlinear filtering are the Kallianpur-Striebel formula, the Fujisaki-Kallianpur-Kunita (FKK) equation, the Zakai equation and the Kunita's semigroup versions of FKK and the Zakai equation, all will be developed below. We will also establish the equivalency of FKK, Zakai with their Kunita counter parts in the spirit of Szpirglas \cite{Szpirglas1978}. We call the process $X(t)\in H,t\geq 0$ the signal process and it is governed by the evolution formulated in the earlier sections as:
\begin{equation}
	dX=(-iAX+ J(X))dt +XdM, \label{eqn5.1}
\end{equation}
\begin{displaymath}
	X(0)=x.
\end{displaymath}

The measurement process $Y(t)\in \mR^{N}, t\geq 0$ is defined as
\begin{equation}
	dY=h(X)dt + dZ, \label{eqn5.2}
\end{equation}
where $Z(t)\in \mR^{N}$ is a $N$-dimensional Wiener process which may or may not be correlated with $W(t)$. Let us denote by ${\mathcal F}^{Y}_{t}$ the filtration generated by the sensor data over the time period $0\leq s\leq t$ (sigma algebra generated by the back measurement):
\begin{displaymath}
	{\mathcal F}^{Y}_{t}= \sigma \left \{Y_{s}, 0\leq s\leq t\right \}.
\end{displaymath}
Our goal is to study the time evolution of the nonlinear filter characterized by the conditional expectation $E[f(X(t))\vert {\mathcal F}^{Y}_{t}]$, $ t\geq 0$, where $f$ is some measurable function. It is also the least square best estimate of $f(X(t))$ given the back measurements. We will first consider a special case where the signal noise and observation noise are independent. The following Bayes formula known as the Kallianpur-Striebel formula can be derived following \cite{Kallianpur1980}: Let $g$ be ${\mathcal F}^{X}_{t}$-measurable function $g(\cdot):H\rightarrow \mR$ integrable on $(\Omega, \Sigma, m)$, $0\leq t\leq T$. We may assume that processes in \ref{eqn5.2} be defined on the product space $(\Omega\times \Omega_{0}, \Sigma\times \Sigma_{0},m\times m_{0})$ where $Z$ is a Wiener process on $(\Omega_{0},\Sigma_{0}, m_{0})$. Then
\begin{equation}
	E_{m}[g\vert {\mathcal F}^{Y}_{t}] =\frac{\vartheta^{Y}(g)}{\vartheta^{Y}(1)} \label{eqn5.3}
\end{equation}
where 
\begin{displaymath}
	\vartheta^{Y}_{t}(g)=\int_{\Omega} g q_{t}dm
\end{displaymath}
with
\begin{displaymath}
	q_{t}=\exp \left \{ \int_{0}^{t} h(X(s))\cdot d Y(s) -\frac{1}{2} \int_{0}^{t}\vert h(X(s))\vert^{2}ds\right \}.
\end{displaymath}
From the Kallianpur-Striebel formula we can derive the nonlinear filtering equations of Zakai and FKK type using It\'o formula for the special case of signal and noise independent. However we can proceed as follows to get the filtering equations for the general case.

Let us define the formal infinitesimal generator ${\mathcal A}$ of the process $X(t)$ using the transition semigroup $P_{t}$ of $X(t)$ in the following way: for a function $f(\cdot):H\rightarrow \mR$, we set $P_{t}f(x)=E[f(X(t))\vert X(0)=x]$. 

Let us consider the measure space $(H, {\mathcal B}(H)) $ where ${\mathcal B}(H)$ is the Borel algebra generated by closed/open subsets of $H$. Let ${\mathcal B}_{b}(H)$ be the class of bounded ${\mathcal B}(H)$-measurable functions on $H$. For $f_{n},f \in {\mathcal B}_{b}(H)$ we say that $f_{n}\rightarrow f$ weakly if $f_{n}\rightarrow f$ pointwise and the sequence $f_{n}$ is uniformly bounded.

Let ${\mathcal D}({\mathcal A})$ be the class of functions $f$ in ${\mathcal B}_{b}(H)$ such that there exists $f_{0}\in {\mathcal B}_{b}(H)$ satisfying
\begin{equation}
	(P_{t}f)(x)=f(x) + \int_{0}^{t}(P_{s}f_{0})(x)ds, \hspace{.1in} \forall x\in H. \label{eqn5.4}
\end{equation}
Note that $f_{0}$ is uniquely determined by the above equation. Hence we set ${\mathcal A}f=f_{0}$ and formally define the extended generator ${\mathcal A}$ as:
\begin{displaymath}
	{\mathcal A}f(x):= \mbox{  weak-limit}_{t\downarrow 0^{+}}\frac{P_{t}f(x)-f(x)}{t}, \hspace{.1in} f\in D({\mathcal A}),
\end{displaymath}
where  
\begin{displaymath}
	D({\mathcal A})= \left \{f\in {\mathcal B}_{b}(H): \mbox{  weak-limit }_{t\downarrow 0^{+}}\frac{P_{t}f(x)-f(x)}{t} = {\mathcal A}f(x) \mbox{  exists} \right\}.	
\end{displaymath}
Since the domain $D({\mathcal A})$ cannot easily be defined explicitly, we will also work with a sub-class of functions known as cylindrical test functions ${\mathcal D}_{cl}$ defined as follows:
\begin{displaymath}
	{\mathcal D}_{cl} :=\left \{f \in C_{b}(H): f(X)=\varphi (\langle X, e_{1}\rangle_{H},\cdots, \langle X, e_{N}\rangle_{H} )\right \},
\end{displaymath}
where $ e_{i}\in D(A), i=1,\cdots, N$ and $ \varphi\in C^{\infty}_{0}(\mR^{N})$.

Then $D_{x}f(x)\in D(A)$ and also ${\mathcal D}_{cl}\subset D({\mathcal A})$. The formal generator of the stochastic process $X(t)$ is given by
\begin{displaymath}
	{\mathcal A}\Phi(x)=\frac{1}{2}\mbox{Tr}(x^{*}QxD^{2}_{x}\Phi(x))+\langle -iAx+J(x),D_{x}\Phi \rangle
\end{displaymath}
\begin{displaymath}
	+\int_{Z}\left (\Phi(x+\Psi(x, z))-\Phi(x) -(D_{x}\Phi, \Psi(x,z))\right )d\nu(z), \hspace{.1in} \Phi\in {\mathcal D}_{cl}. 
\end{displaymath}

We now note that the solvability theorems established in the previous section imply by It\'o's formula the following result.
\begin{Pro}
	Define 
	\begin{equation}
		M_{t}(f)=f(X(t))-f(x)-\int_{0}^{t}{\mathcal A}f(X(s))ds, \hspace{.1in} f\in {\mathcal D}_{cl}, \label{eqn5.6}
	\end{equation}
	where $X(t)$ is the process defined by the stochastic partial differential equation \ref{eqn5.1}. Then $(M_{t}(f), {\mathcal F}_{t}, m)$ is a martingale.
\end{Pro}
Nonlinear filtering for infinite dimensional semimartingale with specific application to stochastic Navier-Stokes equation forced by Levy noise has been developed in \cite{Fernando2013}. Here we will sketch the derivation for continuous semimartingales but the final filtering equations will be given for the jump semimartingale case. In fact applying It\'o formula \cite{DaPrato2014} to the stochastic process defined by \ref{eqn5.1}
\begin{displaymath}
	f(X(t))=f(x)+\int_{0}^{t}{\mathcal A}f(X(s))ds	+\int_{0}^{t}(\frac{\partial f}{\partial x}(X(s)), X(s)dM(s)), \hspace{.1in} f\in {\mathcal D}_{cl}.
\end{displaymath}
Let us now state a general martingale lemma and then specialize it to the filtering problem.
\begin{Lem} Let ${\mathcal F}_{t}$ and ${\mathcal G}_{t}$ be filtrations with ${\mathcal G}_{t}\subset {\mathcal F}_{t}$. Suppose that 
	\begin{displaymath}
		M_{t}^{\mathcal F}=U(t)-U(0)-\int_{0}^{t}V(s)ds
	\end{displaymath}
	is an ${\mathcal F}_{t}$-martingale. Then
	\begin{displaymath}
		M^{\mathcal G}_{t}=	E[U(t)\vert {\mathcal G}_{t}]-E[U(0)\vert {\mathcal G}_{0}]-\int_{0}^{t}E[V(s)\vert {\mathcal G}_{s}]ds
	\end{displaymath}
	is a ${\mathcal G}_{t}$-martingale.
\end{Lem}
We now state this property in the nonlinear filtering context:
\begin{Thm}
	Define
	\begin{equation}
		M_{t}^{Y}(f)=E[f(X(t))\vert {\mathcal F}_{t}^{Y}]-f(x) -\int_{0}^{t}E[{\mathcal A}f(X(s))\vert {\mathcal F}_{s}^{Y}]ds, \hspace{.1in} f\in {\mathcal D}_{cl}.  \label{eqn5.7}
	\end{equation}	
	Then $(M_{t}^{Y}(f), {\mathcal F}_{t}^{Y}, m)$ is a martingale.
\end{Thm}

\begin{proof}
	First note that
	\begin{displaymath}
		E\left [M^{Y}_{f}(f)\vert {\mathcal F}^{Y}_{s}\right ]-M^{Y}_{s}(f)=E\left [M^{Y}_{f}(f)-M^{Y}_{s}(f)\vert {\mathcal F}^{Y}_{s}\right ].	
	\end{displaymath}
	Noting that ${\mathcal F}_{t}^{Y}\subset {\mathcal F}_{t}$, we consider, for $s<t$,
	\begin{displaymath}
		E\left [M^{Y}_{f}(f)-M^{Y}_{s}(f)\vert {\mathcal F}^{Y}_{s}\right ]
	\end{displaymath}
	\begin{displaymath}
		= E\left [E[f(X(t))\vert {\mathcal F}_{t}^{Y}]-E[f(X(s))\vert {\mathcal F}_{s}^{Y}]\vert {\mathcal F}^{Y}_{s}\right ]- E\left [\int_{s}^{t}E[{\mathcal A}f(X(r))\vert {\mathcal F}_{r}^{Y}]dr\vert {\mathcal F}^{Y}_{s}\right ].
	\end{displaymath}
	Using the tower property of the conditional expectation we conclude that
	\begin{displaymath}
		E\left [M^{Y}_{f}(f)-M^{Y}_{s}(f)\vert {\mathcal F}^{Y}_{s}\right ]
	\end{displaymath}
	\begin{displaymath}
		= E[f(X(t))\vert {\mathcal F}_{s}^{Y}]-E[f(X(s))\vert {\mathcal F}_{s}^{Y}]-  \int_{s}^{t}E[{\mathcal A}f(X(r))\vert {\mathcal F}_{s}^{Y}]dr
	\end{displaymath}
	\begin{displaymath}
		=E\left [ f(X(t))-f(X(s))-\int_{s}^{t}{\mathcal A}f(X(r))dr \vert {\mathcal F}_{s}^{Y}\right ]
	\end{displaymath}
	\begin{displaymath}
		= E\left [M_{t}(f)-M_{s}(f)\vert {\mathcal F}^{Y}_{s}\right]= E\left [E\left [M_{t}(f)-M_{s}\vert{\mathcal F}_{s}\right] \vert {\mathcal F}^{Y}_{s}\right ]=0,
	\end{displaymath}
	since $M_{t}(f)$ is a $({\mathcal F}_{t},m)$- martingale. Hence $E\left [M^{Y}_{f}(f)\vert {\mathcal F}^{Y}_{s}\right ]= M^{Y}_{s}(f)$.
\end{proof}

Fujisaki-Kallianpur-Kunita (1972) \cite{Fujisaki1972}, in particular Theorem 3.1, Lemma 4.2, Theorem 4.1 and equation 6.11 in that paper, allows us to characterize $M_{t}^{Y}$ explicitly using the next set of arguments.

\begin{Def} Let the innovation process be:
	\begin{equation}
		d\nu^{Y}(t)=  dY(t)-E[h(X(t))\vert {\mathcal F}_{t}^{Y}]dt, \hspace{.1in} t\in [0,T]. \label{eqn5.8}
	\end{equation}	
	
\end{Def} 
\begin{Lem} \cite{Fujisaki1972}  
	The innovation process $(\nu^{Y}, {\mathcal F}^{Y}_{t},m)$ is an $N$-vector standard Wiener process. Furthermore ${\mathcal F}^{Y}_{s}$ and $\sigma\left \{\nu_{v}^{Y}-\nu_{u}^{Y}; s\leq u\leq v \leq T\right \}$  are independent.
\end{Lem}

Let us now state an important martingale representation theorem (Theorem 3.1 \cite{Fujisaki1972}):
\begin{Thm}
	Every square integrable martingale $(M^{Y}_{t},{\mathcal F}^{Y}_{t},m)$ is sample continuous and has the representation
	\begin{equation}
		M_{t}^{Y}-M^{Y}_{0}=\int_{0}^{t}\Xi(s)\cdot d\nu^{Y}(s).	\label{eqn5.9}
	\end{equation}
	where $\nu(t)$ is the innovation process and $\Xi(s)=(\Xi_{s}^1, \cdots,\Xi_{s}^{N})	$ is jointly measurable and adapted to ${\mathcal F}^{Y}_{s}$.	
\end{Thm}

\begin{Lem} \cite{Fujisaki1972}  Let $(M_{t}(f), {\mathcal F}_{t},m)$ be the square integrable martingale. Then there exists unique sample continuous process $\langle \langle M(f), Z^{i}\rangle\rangle_{t}, (i=1,\cdots,N)$ adapted to ${\mathcal F}_{t}$ such that almost all sample functions are of bounded variation and $M_{t}(f)-\langle\langle M(f), Z^{i}\rangle\rangle_{t}$ are ${\mathcal F}_{s}$-martingales. Furthermore each $\langle\langle M(f), Z^{i}\rangle\rangle_{t}$ has the following properties: it is absolutely continuous with respect to Lebesgue measure in $[0,T]$. There exists a modification of the Radon-Nikodym derivative which is $(t,\omega)$-measurable and adapted to ${\mathcal F}_{t}$ and which shall denote by $\tilde{D}_{t}^{i}f(\omega)$. Then using the vector notation $\tilde{D}_{t}f=(\tilde{D}_{t}^{1}f,\cdots, \tilde{D}_{t}^{N}f)$,
	\begin{displaymath}
		\langle\langle M(f),Z \rangle\rangle_{t} =\int_{0}^{t}\tilde{D}_{s}fds, \hspace{.1in} \mbox{ a.s},
	\end{displaymath}
	where
	\begin{displaymath}
		\int_{0}^{T}E\vert \tilde{D}_{s}f\vert^{2}ds <\infty.
	\end{displaymath}
	If the process $X(t)$ and $Z(t)$ are independent then 
	\begin{displaymath}
		\langle\langle M(f),Z \rangle\rangle_{t}=0, \mbox{ a.s}.
	\end{displaymath}
\end{Lem}

The following theorem characterizes $\Xi$ explicitly:
\begin{Thm}
	Let $f\in D({\mathcal A})$ and 
	\begin{equation}
		\int_{0}^{T}\vert h(X(t))\vert^{2}dt<\infty, \hspace{.1in} \mbox{ a.s.}\label{eqn5.10}
	\end{equation}
	Then the evolution of the conditional expectation $	E[f(X(t))\vert {\mathcal F}_{t}^{Y}]$ (the nonlinear filter) is characterized by the Fujisaki-Kallianpur-Kunita equation:
	\begin{displaymath}
		E[f(X(t))\vert {\mathcal F}_{t}^{Y}]=	E[f(X(0))\vert {\mathcal F}_{0}^{Y}] + \int_{0}^{t}E[{\mathcal A}f(X(s))\vert {\mathcal F}_{s}^{Y}]ds 
	\end{displaymath}
	\begin{displaymath}
		+	\int_{0}^{t} \left \{E[f(X(s))h(X(s))\vert {\mathcal F}_{s}^{Y}] -E[f(X(s))\vert {\mathcal F}_{s}^{Y}]E[h(X(s))\vert {\mathcal F}_{s}^{Y}]	\right \}\cdot d\nu^{Y}(s)
	\end{displaymath}
	\begin{equation}	
		+	\int_{0}^{t} \left \{ E[\tilde{D}_{s}f(X(s))\vert {\mathcal F}_{s}^{Y}]\right \}\cdot d\nu^{Y}(s), \label{eqn5.11}
	\end{equation}
	where  $\tilde{D}_{s}f(X(s))$ is given by
	\begin{displaymath}
		\langle\langle M(f),Z \rangle\rangle_{t} = \int_{0}^{t}\tilde{D}_{s}f ds.
	\end{displaymath}
	In particular $\tilde{D}_{t}f =0$ if $W$ and $Z$ are independent. Moreover, if 
	\begin{displaymath}
		BW=\sigma_{1}W_{1}+\sigma_{2}Z,
	\end{displaymath}
	then 
	\begin{displaymath}
		\tilde{D}_{t}f=\sigma_{2}^{*}\frac{\partial f}{\partial x}.
	\end{displaymath}
	
\end{Thm}

Note that if $h$ has a growth condition
\begin{displaymath}
	\vert h(x)\vert \leq C \vert | x\vert |,  \hspace{.1in} x\in H,
\end{displaymath}
then by theorem \ref{Thm:3.1} the condition \ref{eqn5.10} is satisfied.

\subsubsection{Existence and Uniqueness of Measure-valued solutions to Filtering Equations}

A theorem of Getoor \cite{Getoor1975} provides the existence of a random measure $\pi_{t}^{Y}$ that is measurable with respect to ${\mathcal F}_{t}^{Y}$ such that 
\begin{displaymath}
	E[f(X(t))\vert {\mathcal F}_{t}^{Y}]=\pi_{t}^{Y}[f] =\int_{H}f(\zeta)\pi_{t}^{Y}(d\zeta).
\end{displaymath}

Substituting in \ref{eqn5.11}, the probability measure-valued process $\pi_{t}^{Y}[\cdot]$ satisfies the  Fujisaki-Kallianpur-Kunita (FKK) type equation: 
\begin{equation}
	d\pi_{t}^{Y}[f]=\pi_{t}^{Y}[{\mathcal A}f]dt +\left(\pi_{t}^{Y}[hf+\tilde{D}_{t}f]-\pi_{t}^{Y}[h]\pi_{t}^{Y}[f]\right)\cdot d\nu^{Y}(t), \hspace{.1in} \mbox{ for } f\in {\mathcal D}_{cl}. \label{eqn5.12}
\end{equation}	
If we set 
\begin{equation}
	\vartheta_{t}^{Y}[f]:= \pi_{t}^{Y}[f]\exp\left \{\int_{0}^{t}\pi_{t}^{Y}[h]\cdot dY(s) -\frac{1}{2}\int_{0}^{t}\vert \pi_{t}^{Y}[h]\vert^{2}ds\right \},\label{eqn5.13}
\end{equation}	
then the measure-valued process  $\vartheta_{t}^{y}[\cdot]$ satisfies the Zakai type equation:
\begin{equation}
	d\vartheta_{t}^{Y}[f]=\vartheta_{t}^{Y}[{\mathcal A}f]dt +\vartheta_{t}^{Y}[hf+\tilde{D}_{t}f]\cdot dY(t), \mbox{ for } f\in {\mathcal D}_{cl}. \label{eqn5.14}
\end{equation}	
Such measure-valued evolutions were first derived by Kunita \cite{Kunita1971} in the context of compact space valued signal processes.

One method of proving uniqueness of measure-valued solutions to filtering equations is to start with the unique solution of the backward Kolmogorov equation similar to the papers in the case of nonlinear filtering of stochastic Navier-Stokes equation \cite{Sritharan1995} or for the case of stochastic reaction diffusion equation \cite{Hobbs1996}.  Here we use the solution $\Phi(t)$ of the backward Kolmogorov equation with initial data $\Phi$ and express the Zakai equation as:
\begin{equation} 
	\vartheta_{t}^{Y}[\Psi]=\vartheta_{0}^{Y}[\Phi(t)]+ \int_{0}^{t} \vartheta_{s}^{Y}[ (\frac{\partial}{\partial s} +{\mathcal A})\Phi(s)]ds 
	+ \int_{0}^{t} \left(\vartheta_{s}^{Y}[(h+\tilde{D}_{t})\Phi(s)]\right)\cdot d\nu_{s}, \label{eqn5.15}
\end{equation}	
and utilize the fact the $\Phi$ satisfies $(\frac{\partial}{\partial s} +{\mathcal A})\Phi(s)=0$ in the above equation (eliminating the integral involving the generator ${\mathcal A})$ and then proceed to prove the uniqueness of the random measures $\vartheta_{t}^{Y}$ as well as  $\pi_{t}^{Y}$. 

We now specialize to the case of signal noise $W$ and observation noise $Z$ being independent. In this setting we can use the results of Kunita \cite{Kunita1971} and Szpirglas\cite{Szpirglas1978} to obtain equivalent nonlinear filtering equation that do not explicitly involve the generator ${\mathcal A}$ of the signal Markov process $X$.

Uniqueness theorem relies on the following lemma \cite{Szpirglas1978, Kallianpur1988}:
\begin{Lem}
	Let us define a subclass of ${\mathcal D}({\mathcal A})$ as:
	\begin{displaymath}
		{\mathcal D}_{2}({\mathcal A}):= \left \{ f\in {\mathcal D}({\mathcal A}): {\mathcal A}f \in {\mathcal D}({\mathcal A}) \right \}.	
	\end{displaymath}
	Let $\mu_{1},\mu_{2}\in {\mathcal M}(H)$ be finite measures on $(H, {\mathcal B}(H))$ such that
	\begin{equation}
		\langle \mu_{1}, f\rangle =\langle \mu_{2}, f\rangle, 	\hspace{.2in} \forall f\in {\mathcal D}_{2}({\mathcal A}), \label{eqn5.16}
	\end{equation}
	where $\langle \mu, f\rangle = \int_{H} f(x)d \mu(x).$ Then $\mu_{1}=\mu_{2}\in {\mathcal M}(H)$.
\end{Lem}
The uniqueness of measure-valued solutions for the nonlinear filtering equations and equivalence of two different forms, one involving the formal generator and other using Feller semigroup of the system Markov process eliminating the term involving the generator started with Kunita\cite{Kunita1971} and Szpirglas \cite{Szpirglas1978} and the following series of theorems can be proven by the same methods as in those original works for signal state space compact, locally compact and general Hilbert space cases as in \cite{Kunita1971,Fernando2013}: 

\begin{Thm}
	The random probability measure valued process $\pi^{Y}_{t} \in {\mathcal P}(H)$ uniquely solves the evolution equation:
	\begin{equation}
		\pi_{t}^{Y}[f]=\pi_{0}^{Y}[f]+\int_{0}^{t}\pi_{s}^{Y}[{\mathcal A}f]ds + \int_{0}^{t}\left(\pi_{s}^{Y}[hf]-\pi_{s}^{Y}[h]\pi_{s}^{Y}[f]\right)\cdot d\nu^{Y}(s), \hspace{.05in}  \forall f\in D({\mathcal A}). \label{eqn5.17}
	\end{equation}
	Equivalently, $\pi^{Y}_{t}$ uniquely solves the evolution equation:
	\begin{equation} 
		\pi_{t}^{Y}[f]=\pi_{0}[P_{t}f]+ \int_{0}^{t} \left(\pi_{s}^{Y}((P_{t-s}f) h)-\pi_{s}^{Y}(P_{t-s}f) \pi_{s}^{Y}(h)\right)\cdot d\nu^{Y}_{s},	\hspace{.05in}  \forall f\in {\mathcal B}_{b}(H). \label{eqn5.18}
	\end{equation}
	
\end{Thm}

\begin{Thm}
	the random measure valued process $\vartheta^{Y}_{t} \in {\mathcal M}(H)$ uniquely solves the evolution equation:
	\begin{equation}
		\vartheta_{t}^{Y}[f]=\vartheta_{0}^{Y}[f]+\int_{0}^{t}\vartheta_{s}^{Y}[{\mathcal A}f]ds + \int_{0}^{t}\vartheta_{s}^{Y}[hf]\cdot dY(s), \hspace{.05in} \forall f\in D({\mathcal A}). \label{5.19}
	\end{equation}
	Equivalently, $\vartheta^{Y}_{t} \in {\mathcal M}(H)$ uniquely solves the evolution equation:
	\begin{equation} 
		\vartheta_{t}^{Y}[f]=\vartheta_{0}[P_{t}f]+ \int_{0}^{t} \vartheta_{s}^{Y}((P_{t-s}f) h)\cdot dY({s}),	\hspace{.05in} \forall f\in {\mathcal B}_{b}(H). \label{5.20}
	\end{equation}
	
\end{Thm}

\begin{Thm}
	The  FKK equation (\ref{eqn5.18}) for arbitrary initial condition $\pi_{0}$ has a unique solution $\pi^{Y}_{t}$. Furthermore,
	\begin{enumerate}
		\item The solution $\pi^{Y}_{t}$ is $\sigma(Y(s)-Y(0); 0\leq s\leq t)\vee\sigma(\pi_{0})$ -measurable.
		\item Let $\pi_{t}^{Y(\nu)}$ and $\pi_{t}^{Y(\mu)}$ be solutions with initial conditions $\pi_{0}=\nu$ and $\pi_{0}=\mu$ respectively, where $\mu,\nu \in {\mathcal M}(H)$.  
		Then for every $t>0$ 
		\begin{equation}
			\lim_{\nu\rightarrow \mu}E\left [ \vert\pi_{t}^{Y(\nu)}(f)-\pi_{t}^{Y(\mu)} (f)\vert \right] =0, \hspace{.1in} \forall f \in C_{b}(H).\label{eqn5.21}
		\end{equation}
	\end{enumerate}
\end{Thm}

\begin{Thm}
	The filtering process $(\pi^{Y}, {\mathcal F}_{t}^{Y}, P_{\mu})$, $\mu \in {\mathcal M}(H)$ are Markov process associated with the transition probabilities $\Pi_{t}(\nu, \Gamma)$ defined by 
	\begin{displaymath}
		\Pi_{t}(\nu,\Gamma)=P\{\pi_{t}^{Y}\in \Gamma : \hspace{.1in} \Gamma \in {\mathcal B} ({\mathcal M}(H))\},
	\end{displaymath}
	where ${\mathcal B} ({\mathcal M}(H))$ is the Borel algebra generated by the open (or closed) sets in ${\mathcal M}(H)$.  
	
	Furthermore, the transition probabilities $\Pi_{t}(\nu,\Gamma)$ define a Feller semigroup in $C_{b}({\mathcal M}(H))$, where $C_{b}({\mathcal M}(H))$ is the space of all real continuous functions over ${\mathcal M}(H)$.
\end{Thm}
\subsection{Evolution Equation for Error Covariance: It\^o calculus case }

For linear systems the paper of Kalman and Bucy \cite{Kalman1961} gives a characterization of the evolution of error covariance using a Riccati equation. This concept was generalized by Kunita\cite{Kunita1991} for nonlinear filters. Let us describe the time evolution of error covariance slightly generalizing the results of Kunita\cite{Kunita1991}. The error covariance of general moments $f(X(t))$ and $g(X(t))$ is given by
\begin{equation}
	{\mathcal P}_{t}[f,g]=E\left [\left (f(X(t))-\pi_{t}(f)\right )\left(g(X(t))-\pi_{t}(g)\right )\right]. \label{eqn5.25}
\end{equation}
\begin{Thm} Let $f,g\in D({\mathcal A})$. Then ${\mathcal P}_{t}[f,g]$ is differentiable with respect to $t$ and the derivative satisfies
	\begin{displaymath}
		\frac{d}{dt} {\mathcal P}_{t}[f,g]= {\mathcal P}_{t}[f,{\mathcal A}g] + {\mathcal P}_{t}[{\mathcal A}f,g]+{\mathcal Q}_{t}[f,g] 
	\end{displaymath}
	\begin{equation}
		-	\sum_{i=1}^{N}\int_{0}^{t} E \left [ \Theta_{i}(s)\right ]ds, \label{eqn5.26}
	\end{equation}
	where
	\begin{displaymath}
		\Theta_{i}= \pi_{s} \left \{(f-\pi_{s}(f))(h^{i}-\pi_{s}(h^{i}))+\tilde{D}f\right \} \pi_{s}\left \{  (g-\pi_{s}(g))(h^{i}-\pi_{s}(h^{i}))+\tilde{D}g\right\},
	\end{displaymath}
	and	
	\begin{equation}
		{\mathcal Q}_{t}[f,g] =E\left [\mbox{ tr} (x^{*}Q x D_{x}f \otimes D_{x}g)\right].	 \label{eqn5.27}
	\end{equation}
\end{Thm}

\begin{proof}
	
	Note that we have:
	\begin{displaymath}
		{\mathcal A}(fg)=f{\mathcal A}g + g{\mathcal A}f +\mbox{ tr} (x^{*}Q x D_{x}f \otimes D_{x}g).
	\end{displaymath}
	We first write the FKK equation for $\pi_{t}(fg)$ as
	\begin{displaymath}
		\pi_{t}(fg)=\pi_{0}(fg) +\int_{0}^{t} \pi_{s}(f {\mathcal A}g)ds +\int_{0}^{t}\pi_{s}(g{\mathcal A}f)ds
	\end{displaymath}
	\begin{displaymath}
		+\int_{0}^{s}\pi_{s}( \mbox{ tr} (x^{*}Q x D_{x}f \otimes D_{x}g))ds +M^{\pi}_{t},
	\end{displaymath}
	where $M^{\pi}_{t}$ is a ${\mathcal F}^{Y}_{t}$-martingale with mean zero and $f,g\in D({\mathcal A})$. Similarly since $\pi_{t}(f)$ and $\pi_{t}(g)$ satisfy FKK equation we have by Ito formula for the product $\pi_{t}(f)\pi_{t}(g)$:
	\begin{displaymath}
		\pi_{t}(f)\pi_{t}(g)=\pi_{0}(f)\pi_{0}(g)+ \int_{0}^{t}\pi_{s}(f)\pi_{s}({\mathcal A}g)ds+\int_{0}^{t}\pi_{s}(g)\pi_{s}({\mathcal A}f)ds+
	\end{displaymath}
	\begin{displaymath}
		\sum_{i}\int_{0}^{t} \left \{ \pi_{s}(fh^{i})-\pi_{s}(f)\pi_{s}(h^{i})+\pi_{s}(\tilde{D}f)\right \}\left \{ \pi_{s}(gh^{i})-\pi_{s}(g)\pi_{s}(h^{i})+\pi_{s}(\tilde{D}g)\right \}ds 
	\end{displaymath}
	\begin{displaymath}
		+\tilde{M}^{\pi}_{t},
	\end{displaymath}
	where $\tilde{M}^{\pi}_{t}$ is a ${\mathcal F}^{Y}_{t}$-martingale with mean zero. Now noting that
	\begin{displaymath}
		\pi_{s}(fg)-\pi_{s}(f)\pi_{s}(g)=\pi_{s}\left \{(f-\pi_{s}(f))(g-\pi_{s}(g))\right\},
	\end{displaymath}
	we get 
	\begin{displaymath}
		{\mathcal P}_{t}(f,g)=E[\pi_{t}(fg)-\pi_{t}(f)\pi_{t}(g)].
	\end{displaymath}
	Substituting for $\pi_{t}(fg)$ and $\pi_{t}(f)\pi_{t}(g)$ from  the previous two equations and taking expectation we arrive at
	\begin{displaymath}
		{\mathcal P}_{t}(f,g)={\mathcal P}_{0}(f,g) +\int_{0}^{t}\left \{ {\mathcal P}_{s}[f,{\mathcal A}g] + {\mathcal P}_{s}[{\mathcal A}f,g]+{\mathcal Q}_{s}[f,g] \right \}ds -\sum_{i=1}^{N}\int_{0}^{t} E \left [ \Theta_{i}\right ]ds,	
	\end{displaymath}
	where
	\begin{displaymath}
		\Theta_{i}= \pi_{s} \left \{(f-\pi_{s}(f))(h^{i}-\pi_{s}(h^{i}))+\tilde{D}f\right \} \pi_{s}\left \{  (g-\pi_{s}(g))(h^{i}-\pi_{s}(h^{i}))+\tilde{D}g\right\}.
	\end{displaymath}
\end{proof}

\section{White Noise Theory of Nonlinear Filtering}
White noise theory of nonlinear filtering was initiated by A. V. Balakrishnan \cite{ Bala1980} and systematically developed by G. Kallianpur and R. Karandikar \cite{ Kallianpur1988}. In \cite{Sritharan2023a} we developed such filtering method for classical and quantum spin systems. In this paper we will further develop this method to be applicable to the nonlinear stochastic wave equations discussed in this paper. Key mathematical equations of white noise nonlinear filtering theory are essentially parallel to stochastic calculus counterpart but of deterministic in nature. The Kallianpur-Striebel formula, the Fujisaki-Kallianpur-Kunita equation, the Zakai equation and the Kunita's semigroup versions of FKK and the Zakai equation all have white noise counterparts and will be developed below for the nonlinear stochastic wave equations. We will also establish the equivalency of FKK, Zakai with their Kunita counter parts in the spirit of Szpirglas\cite{Szpirglas1978}. 

We will now describe the theory of Segal \cite{Segal1956} and Gross \cite{Gross1960, Gross1962} of finitely additive cylindrical measures on separable Hilbert spaces. Let ${\mathcal H}$ be a separable Hilbert space and ${\mathcal P}$ be the set of orthogonal projections on ${\mathcal H}$ having finite dimensional range. For $P\in {\mathcal P}$, let ${\mathcal C}_{P}=\{P^{-1}B: B \mbox{ a Borel set in the range of } P\}$. Let ${\mathcal C}=\cup_{P\in {\mathcal P}} {\mathcal C}_{P}$. A cylinder measure $\bn$ on ${\mathcal H}$ is a finitely additive measure on $({\mathcal H}, {\mathcal C})$ such that its restriction to ${\mathcal C}_{P}$ is countably additive for each $P\in {\mathcal P}$.

Let $L$ be a representative of the weak-distribution corresponding to the cylinder measure $\bn$. This means that $L$ is a linear map from ${\mathcal H^{*}}$ (identified with ${\mathcal H}$) in to ${\mathcal L}_{0}(\Omega_{1}, \Sigma_{1},m_{1})$ -the space of all random variables on a $\sigma$-additive probability space $(\Omega_{1}, \Sigma_{1},m_{1})$ such that 
\begin{displaymath}
	\bn(h\in {\mathcal H}: (\langle h,h_{1}\rangle, \langle h,h_{2}\rangle,\cdots, \langle h,h_{k}\rangle)\in B)
\end{displaymath}
\begin{equation}
	= m_{1}(\omega\in \Omega_{1}; \left(L(h_{1})(\omega), L(h_{2})(\omega), \cdots, L(h_{k})(\omega)\right)\in B) \label{eqn6.1}
\end{equation}
for all Borel sets $B$ in $\mR^{k}$, $h_{1}, \cdots, h_{k}\in {\mathcal H}$, $k\geq 1$. Two maps $L, L'$ are said to be equivalent if both satisfy \ref{eqn6.1} and the equivalence class of such maps is the weak distribution corresponding to the cylindrical measure $\bn$.

A function $f$ on ${\mathcal H}$ is called a tame function (or cylindrical function) if it is of the form
\begin{equation}
	f(y)=\varphi(\langle y,h_{1}\rangle,\cdots, \langle y,h_{k}\rangle ), \mbox{ for some } k\geq 1, h_{1}, \cdots, h_{k}\in {\mathcal H}  \label{eqn6.2}
\end{equation}	
and a Borel function $\varphi:\mR^{k}\rightarrow \mR$. For a tame function $f$ given by \ref{eqn6.2}, we associate the random variable $\varphi(L(h_{1}), \cdots, L(h_{k}))$ on $(\Omega_{1},\Sigma_{1},m_{1})$ and denote it by $\tilde{f}$. We can extend this map $f\rightarrow \tilde{f}$ to a larger class of functions as follows \cite{Gross1962}:
\begin{Def} \label{D6.1} Let ${\mathcal L}({\mathcal H}, {\mathcal C}, \bn)$ be the class of continuous functions $f$ on ${\mathcal H}$ such that the net $\{\tilde{(f\circ P)}: P\in {\mathcal P}\}$ is Cauchy in $m_{1}$-measure. Here $P_{1}<P_{2}$ if the range $P_{1} \subseteq$ range $P_{2}$. For $f\in {\mathcal L}({\mathcal H}, {\mathcal C}, \bn)$, we define:
	\begin{displaymath}
		\tilde{f}={\mbox{Lim in Prob}}_{P\in {\mathcal P}} \tilde{(f\circ P)	}.
	\end{displaymath}
	The map $f\rightarrow \tilde{f}$ is linear, multiplicative and depends only on $f$ and $\bn$ and is independent of the representative of the weak distribution.	
\end{Def}

\begin{Def} \label{D6.2} The function $f\in {\mathcal L}({\mathcal H}, {\mathcal C}, \bn)$ is integrable with respect to $\bn$ if $\int_{\Omega_{1}}\vert \tilde{f}\vert dm_{1}<\infty$ and then for $C\in {\mathcal C}$ define the integral $f$ with respect to $\bn$ over $C$ denoted by $\int_{C}fd\bn$ by
	\begin{displaymath}
		\int_{C}fd{\boldmath n}=\int_{\Omega_{1}} \tilde{1_{C}} \tilde{f} dm_{1}.		
	\end{displaymath}
\end{Def}	
Let us now consider the class of all Gauss measures on ${\mathcal H}$ defined by
\begin{displaymath}
	\mu\{y\in {\mathcal H}: \langle y,h\rangle \leq a\}=\frac{1}{\sqrt{2\pi v(h)}}\int_{-\infty}^{a} \exp{(-\frac{x^{2}}{2v(h)})}dx, \hspace{.1in} \forall h\in {\mathcal H}.
\end{displaymath}
The special case of $v(h)=\|h\|^{2}$ is called the canonical Gauss measure $\bm$. The following Lemma from Sato \cite{Sato1969} (in particular, Lemma 6 in that paper)clarifies the $\sigma$-additivity of Gauss measures on separable Hilbert spaces.	
\begin{Lem} \label{L6.1} let ${\mathcal H}$ be a separable Hilbert space and let $\mu$ be a Gaussian cylinder measure on $({\mathcal H}, {\mathcal C})$ with variance $v(h)$. Then $\mu$ has a $\sigma$-additive extension to $({\mathcal H}, \bar{\mathcal C})$ (here $\bar{\mathcal C}$ is the minimal $sigma$-algebra containing ${\mathcal C})$ if and only if the characteristic functional of $\mu$ is of the form:
	\begin{displaymath}
		\int_{\mathcal H}e^{i \langle h,x\rangle }d\mu(x)	=\exp {\left [ i\langle h,m_{e}\rangle -\frac{1}{2}\|Sh\|^{2}\right]}, \hspace{.1in} \forall h\in {\mathcal H},
	\end{displaymath}
	where $	m_{e}$ is an element of ${\mathcal H}$, and $S$ is a nonnegative selfadjoint Hilbert-Schmidt operator on ${\mathcal H}$.
\end{Lem}
The canonical Gauss measure $\bm$ corresponds to $m_{e}=0$ and $S=I$ and since the identity operator $I$ is not Hilbert-Schmidt, the cannonical Gauss measure $\bm$ is only finitely additive (See also Chapter-I, Proposition 4.1 of Kuo\cite{Kuo1975}). 

The identity map $e$ on ${\mathcal H}$, considered as a map from $({\mathcal H}, {\mathcal C}, \bn)$ in to $({\mathcal H}, {\mathcal C})$ is called the Gaussian white noise.

Let us now start with the abstract version of the white noise filtering model as formulated by Kallianpur and Karandikar \cite{Kallianpur1988}:
\begin{equation}
	y=\zeta + e \label{eqn6.3}
\end{equation}
where $\zeta$ is an ${\mathcal H}$-valued random variable defined on a countably additive probability space $(\Omega, {\mathcal F}, m)$, independent of $e$. To formulate this mathematically precisely let ${\mathcal E}={\mathcal H}\times \Omega$ and
\begin{displaymath}
	\Sigma=\cup_{P\in {\mathcal P}} {\mathcal C}_{P}\otimes {\mathcal F}  \mbox{ where } {\mathcal C}_{P}\otimes {\mathcal F} \mbox{ is the product sigma field}.
\end{displaymath}
For $P\in {\mathcal P}$, let $\alpha_{P}$ be the product measure $(\bm\vert {\mathcal C}_{P})\otimes m$ which is countably additive. This defines a unique finitely additive probability measure $\alpha$ on $({\mathcal E}, \Sigma)$ such that $\alpha=\alpha_{P}$ on ${\mathcal C}_{P}\otimes {\mathcal F} $.

Let $e, \zeta, y$ be ${\mathcal H}$-valued maps on ${\mathcal E}$ defined by
\begin{displaymath}
	e(h, \omega)=h,
\end{displaymath}
\begin{displaymath}
	\zeta(h,\omega)=\zeta(\omega),
\end{displaymath}
\begin{equation}
	y(h,\omega)=e(h,\omega)+ \zeta(h,\omega), \hspace{.1in} (h,\omega)\in {\mathcal H}\times \Omega. \label{eqn6.4}
\end{equation}
The model \ref{eqn6.4} is the abstract version of the white noise filtering on $({\mathcal E}, \Sigma, \alpha)$. Our goal is to first characterize the conditional expectation $E[g\vert y]$ in this finitely additive setting. 
\begin{Def}
	If there exists a $v\in {\mathcal L}({\mathcal H}, {\mathcal C}, \bn)$ such that 
	\begin{equation}
		\int_{{\mathcal H}\times \Omega}g(\omega) 1_{C}(y(h,\omega))d\alpha(h,\omega)=\int_{C}v(y)d\bn(y),\hspace{.1in} \forall C\in {\mathcal C},\label{eqn6.5}
	\end{equation}
	then we define $v:=E[g\vert y].$
\end{Def}

Note that the integrand $g(\omega) 1_{C}(y(h,\omega))$ is ${\mathcal C}_{P}\otimes {\mathcal F}$-measurable for $C\in {\mathcal C}_{P}$ and hence the integrals in \ref{eqn6.5} are well-defined.

Let us now state the elegant Bayes formula for the finitely additive white noise framework proved in \cite{Kallianpur1988}:
\begin{Thm} \label{Thm6.5}
	Let $y, \zeta$ be as in \ref{eqn6.4}. Let $g$ be an integrable function on $(\Omega, {\mathcal F}, m)$. Then 
	\begin{equation}
		E[g\vert y] =\frac{\int_{\Omega}g(\omega)exp{\left \{\langle y,\zeta(\omega)\rangle-\frac{1}{2}\|\zeta(\omega)\|^{2}\right\}}dm(\omega)}{\int_{\Omega}exp{\left \{\langle y,\zeta(\omega)\rangle -\frac{1}{2}\|\zeta(\omega)\|^{2}\right\}}dm(\omega)}. \label{eqn6.6}
	\end{equation}
\end{Thm}
We now return to our stochastic wave equation \ref{eqn:1.1} where the state variable $X(t)$ is an $(H, {\mathcal B}(H))$-valued stochastic process (where ${\mathcal B}(H)$ is the Borel algebra generated by the closed or open subsets of $H$) that is defined on the probability space $(\Omega, {\mathcal F},m)$. Let ${\mathcal K}$ be a separable Hilbert space (including $\mR^{N}$) and let the observation vector $\zeta=\zeta(X(t))$ be a measurable function from $H\rightarrow {\mathcal K}$ such that 
\begin{equation}
	\int_{0}^{T}\|\zeta(X(t))\|_{\mathcal K}^{2}dt<\infty, \mbox{ a.s.}. \label{eqn6.7}
\end{equation}
The Hilbert space ${\mathcal H} $ we started with in the abstract white noise formulation will be in this case ${\mathcal H}=L^{2}(0,T;{\mathcal K})$.

The white noise sensor measurement model:

\begin{equation}
	Y(t) =\zeta(X(t))+ e(t) \in {\mathcal K},\hspace{.1in} t>0, \label{equn6.8}
\end{equation}
where $e(t)\in {\mathcal K}$ is a  finite or infinite dimensional white noise. 

Let us define a Borel measure $\rho^{Y}_{t}(\cdot)\in {\mathcal M}(H)$ and a probability measure $\pi_{t}^{Y}(\cdot)\in {\mathcal P}(H)$ as follows:
\begin{displaymath}
	\rho^{Y}_{t}(B)	=E\left [ 1_{B}(X(t))\exp{\int_{0}^{t}\left \{\langle \zeta(X(s)), Y(s)\rangle_ {\mathcal K}-\frac{1}{2}\|\zeta(X(s))\|^{2}_{\mathcal K} \right \}ds}\right ],
\end{displaymath}
for $B\in {\mathcal B}(H)$, and
\begin{displaymath}
	\pi^{Y}_{t}(B)=\frac{\rho^{Y}_{t}(B)}{\rho^{Y}_{t}(H)}.
\end{displaymath}
Then the measures $\rho^{Y}_{t}\in {\mathcal M}(H)$ and $\pi_{t}^{Y}\in {\mathcal P}(H)$ satisfy (by Theorem \ref{Thm6.5}): 
\begin{equation}
	\langle\pi_{t}^{Y},f\rangle=\int_{H}f(x)\pi^{Y}_{t}(dz)=E\left [ f(X(t))\vert Y(s), 0\leq s\leq t \right ], \label{eqn6.9}
\end{equation}

\begin{displaymath}
	\langle\pi_{t}^{Y}, f\rangle = \frac{\langle\rho^{Y}_{t}, f\rangle}{\langle\rho_{t}^{Y},1\rangle},
\end{displaymath}

\begin{displaymath}
	\langle\rho_{t}^{Y}, f\rangle= E\left \{ f(X(t)) \exp{\int_{0}^{t}{\mathfrak{C}}_{s}^{Y}(X(s))ds }\right \},
\end{displaymath}

where  
\begin{displaymath}
	{\mathfrak C}_{s}^{Y}(X)=\langle\zeta(X(s)), Y(s)\rangle_{\mathcal K}-\frac{1}{2}\|\zeta(X(s))\|^{2}_{\mathcal K}.
\end{displaymath}

The following series of theorems involve equivalence of finitely additive nonlinear filtering equations, uniqueness of measure-valued filter, Markov property and robustness. The proofs of these theorems are similar to those of Szpirglas \cite{Szpirglas1978} and those of Kallianpur and Karandikar \cite{Kallianpur1988} and hence not repeated here.

\begin{Thm}
	For $Y\in C([0,T];{\mathcal H})$ and the class of measures $\rho_{t}^{Y}\in {\mathcal M}(H)$, $0\leq t\leq T$ uniquely solves
	\begin{equation}
		\langle \rho_{t}^{Y},f\rangle=\langle\rho_{0},f\rangle+\int_{0}^{t}\langle\rho_{s}^{Y},{\mathcal A}f+{\mathfrak C}_{s}^{Y}f\rangle ds, \hspace{.05in} \forall f\in D({\mathcal A}),\label{eqn6.10}
	\end{equation}	
	or equivalently, $\rho_{t}^{Y}\in {\mathcal M}(H)$, $0\leq t\leq T$ uniquely solves	
	\begin{equation}
		\langle\rho_{t}^{Y},f\rangle=\langle \rho_{0},P_{t}f\rangle+\int_{0}^{t}\langle\rho_{s}^{Y},{\mathfrak C}_{s}^{Y}(P_{t-s}f)\rangle ds, \hspace{.05in}\forall f\in {\mathcal B}_{b}(H). \label{eqn6.11}
	\end{equation}	
	
\end{Thm}

\begin{Thm}
	For $Y\in C([0,T];{\mathcal H})$ the probability measure valued process $\pi_{t}^{Y}\in {\mathcal P}(H)$, $0\leq t\leq T$, uniquely solves
	\begin{displaymath}
		\langle\pi_{t}^{Y},f\rangle=\langle\pi_{0},f\rangle
	\end{displaymath}
	\begin{equation}
		+\int_{0}^{t}\left [\langle \pi_{s}^{Y},{\mathcal A}f 	+{\mathfrak C}_{s}^{Y}f\rangle -\langle \pi_{s}^{Y},{\mathfrak C}_{s}^{Y}\rangle \langle \pi_{s}^{Y},f\rangle\right ] ds,\forall 	 f\in D({\mathcal A}), \label{eqn6.12}
	\end{equation}
	or equivalently $\pi_{t}^{Y}\in {\mathcal P}(H)$ uniquely solves	
	\begin{displaymath}
		\langle\pi_{t}^{Y},f\rangle =\langle\pi_{0},P_{t}f\rangle +
	\end{displaymath}
	\begin{equation}
		\int_{0}^{t}\left [\langle \pi_{s}^{Y},{\mathfrak C}_{s}^{Y}(P_{t-s}f)\rangle-\langle\pi_{s}^{Y},{\mathfrak C}_{s}^{Y}\rangle \langle \pi_{s}^{Y},(P_{t-s}f)\rangle\right ] ds, \forall 	f\in {\mathcal B}_{b}(H). \label{eqn6.13}
	\end{equation}
\end{Thm}	

Let
\begin{displaymath}
	{\mathcal H}_{t}=\left \{ \eta \in {\mathcal H}: \int_{t}^{\infty}\|\eta(r)\|^{2}_{\mathcal K}dr=0\right \}.
\end{displaymath}
Note that ${\mathcal H}_{t}$ is a closed subspace of ${\mathcal H}$. Let $Q_{t}$ be the orthogonal projection onto ${\mathcal H}_{t}$ and let ${\mathcal C}_{Q_{t}}={\mathcal C}({\mathcal H}_{t})$.

\begin{Thm}
	$\pi^{Y}_{t}$ and $\rho^{Y}_{t}$ are respectively ${\mathcal P}(H)$ and ${\mathcal M}(H)$ -valued $\{{\mathcal C}_{Q_{t}}\}$ Markov processes on $({\mathcal H}, {\mathcal C}, \bn)$.
\end{Thm}

\begin{Thm}
	If $Y_{n}\rightarrow Y$ in ${\mathcal H}$ then $\pi_{t}^{Y_{n}}\rightarrow \pi_{t}^{Y}$ and $\rho_{t}^{Y_{n}}\rightarrow \rho_{t}^{Y}$ in total variation norm.
\end{Thm}

\section{First-order Approximation of Nonlinear Filter: Infinite dimensional Kalman filter}

 Let us now consider a linear wave equation obtained by linearizing the stochastic nonlinear wave equation \ref{eqn:1.1} by perturbing about a basic state $X_{0}(t)\in D(A)$. Once again the process $X(t)\in H,t\geq 0$ is the signal process and it is governed by the linear stochastic evolution:
\begin{equation}
	dX(t)={\mathcal A}(t)X(t)dt + B dM, \label{eqn:6.1}
\end{equation}
\begin{displaymath}
	X(0)=x,
\end{displaymath}
where ${\mathcal A}(t)=(-iA+ J'(X_{0}(t)))$.

The measurement process $Y(t)\in \mR^{N}, t\geq 0$ is defined as
\begin{equation}
	dY(t)=h(t) X(t)dt + dZ(t),
\end{equation}
where $Z(t)\in \mR^{N}$ is a $N$-dimensional Wiener process which may or may not be correlated with $W(t)$. The observation vector $h(\cdot): [0,T]\rightarrow {\mathcal L}(H;\mR^{N})$ can be taken as $L^{\infty}$ in time. 

 Let us denote by ${\mathcal F}^{Y}_{t}$ the filtration generated by the sensor data over the time period $0\leq s\leq t$ (sigma algebra generated by the back measurement):
\begin{displaymath}
	{\mathcal F}^{Y}_{t}= \sigma \left \{Y_{s}, 0\leq s\leq t\right \}.
\end{displaymath}
The Theorem \ref{Thm:1.5} assures that the deterministic linearized problem is globally uniquely solvable and linear evolution operator associated with ${\mathcal A}(t)=-iA +J'(X_{0}(t))$ as $ {\mathcal S}(t,\tau)$:
\begin{equation}
	\frac{d {\mathcal S}(t, \tau) x}{dt}= {\mathcal A}(t){\mathcal S}(t, \tau) x,
\end{equation}
is well defined as a bounded two parameter semigroup in $H$. We can express the evolution \ref{eqn:6.1} in the mild form as
\begin{equation}
	X(t)= {\mathcal S}(t,0)x+ \int_{0}^{t}S(t,r)B(r)dM(r),
\end{equation}
and deduce its unique global solvability from Theorem \ref{Thm:3.1}. In the Gaussian case we take as before $M(t)=W(t)\in H$ which is a $H$-valued Wiener process and then the mean:
\begin{displaymath}
	E[X(t)]= {\mathcal S}(t,0)x,
\end{displaymath}
and variance
\begin{displaymath}
	V_{t}= \int_{0}^{t}{\mathcal S}(t,r)B(r)QB^{*}(r){\mathcal S}^{*}(t,r)dr.
\end{displaymath}

Let us define the best least square estimate defined by the conditional expectation $\hat{X}(t)=E [X(t)\vert {\mathcal F}^{Y}_{t}]$ and error covariance
\begin{displaymath}
(P(t)\zeta,\eta)=	E[(X(t)-\hat{X}(t),\zeta) ((X(t)-\hat{X}(t),\eta)  ],\hspace{.1in}  \forall \zeta, \eta\in H.
\end{displaymath}
Then it can be shown that (see chapter 10, \cite{Kallianpur1980} for the stochastic calculus case and chapter XII, \cite{Kallianpur1988} for white noise calculus case with the finite dimensional derivation in those references easily generalize to infinite dimensions) that $\hat{X}(t)$ evolves according to
\begin{equation}
	d\hat{X}(t)=[{\mathcal A}(t)-P(t)h^{*}(t)h(t) ]\hat{X}(t)dt +P(t)h^{*}(t) dY(t)
\end{equation}
for the stochastic calculus case and 
\begin{equation}
	\frac{d\hat{X}}{dt}=[{\mathcal A}(t)-P(t)h^{*}(t)h(t) ]\hat{X}(t)+P(t)h^{*}(t)Y(t)
\end{equation}
for the white noise case.
The error covariance $P(t)$ satisfies the operator Riccati equation:
\begin{equation}
	\frac{d}{dt} \langle P(t)x,y\rangle -\langle P(t)x, {\mathcal A}^{*}(t)y\rangle -\langle {\mathcal A}^{*}(t)x, P(t)y\rangle 
+\langle P(t)h^{*}(t)h(t)P(t)x, y\rangle =\langle F x,y\rangle, \mbox{ a.e.}
\end{equation}
\begin{displaymath}
	P(0)=P_{0} \in {\mathcal L}(H), \hspace{.1in} x, y \in D({\mathcal A}^{*}(t)).
\end{displaymath}

Here $F=B Q\B^{*}$ for the Gaussian case and $F=B \Lambda\B^{*}$ for the pure jump case. Let us now give a general solvability theorem for operator Riccati equation that covers both of the above cases. Direct study of infinite dimensional Riccati equations was initiated by J. L. Lions \cite{Lions1971}, R. Temam \cite{Temam1971}, G. Da Prato \cite{DaPrato1972, DaPrato1973}, L. Tartar \cite{Tartar1974} and V. Barbu \cite{Barbu1976} and also H. J. Kuiper \cite{Kuiper1980, Kuiper1985}. 
\begin{Thm}
Let $P_{0}, F\in {\mathcal L}(H)$ and $P_{0}\geq 0$, i.e., $(P_{0}\zeta, \zeta)\geq 0, \forall \zeta \in H$, then there exists a unique positive strongly continuous global solution $P(\cdot)\in C_{s}([0.T];{\mathcal L}(H))$. Moreover, if $P_{0}, F \in {\mathcal J}_{i}(H), i=1,2$ (respectively trace class or Hilbert-Schmidt) then  $P(\cdot)\in C_{s}([0.T];{\mathcal J}_{i}(H)), i=1,2.$ More generally if $P_{0} \in {\mathcal J}_{p}(H)$ (a trace-ideal) then  $P(\cdot)\in C_{s}([0.T];{\mathcal J}_{p}(H))$.
\end{Thm}

We recall here \cite{Reed1975} that for a bounded operator $A\in {\mathcal L}(H)$, polar decomposition can be expressed as $A=U\vert A\vert $ where $U$ is a partial isometry in $H$ and $\vert A\vert =\sqrt{A^{*}A}$. Trace-class operators ${\mathcal J}_{1}(H)$ is the set $A\in {\mathcal L}(H)$ such that $\mbox{ tr}(\vert A\vert)<\infty.$ We define
\begin{displaymath}
	{\mathcal J}_{p}(H)=\{A\in {\mathcal L}(H): \vert A\vert^{p}\in {\mathcal J}_{1}\}, \mbox{ with norm  } \|A\|_{{\mathcal J}_{p}(H)}= (\mbox{ Tr}(\vert A\vert^{p}))^{1/p}.
\end{displaymath}
For the type of generator ${\mathcal A}(t)$ we have in this paper, the above solvability theorem can be deduced from \cite{DaPrato1972, DaPrato1973} for the case when $P(t)\in {\mathcal L}(H)$, \cite{Bove1974} for the trace class case $P(t)\in {\mathcal J}_{1}(H)$ and \cite{Burns2015} for trace ideal case $P(t)\in {\mathcal J}_{p}(H), p\geq 1$.

We now brifly describe a well-known splitting of the Riccati equation due to S. Chandrasekhar \cite{Chandrasekhar1960} and introduced in filtering theory by Kailath \cite{Kailath1974} for the special case when ${\mathcal A}$ and $h$ are time invariant. Differentiating the operator Riccati equation formally in time we get
\begin{displaymath}
	 \langle \ddot P(t)x,y\rangle -\langle \dot P(t)x, {\mathcal A}^{*}y\rangle -\langle {\mathcal A}^{*}x, \dot P(t)y\rangle 
	+\langle \dot P(t)h^{*}h P(t)x, y\rangle +\langle P(t)h^{*}h\dot P(t)x, y\rangle=0.
\end{displaymath}
Setting $K(t)=P(t)h^{*}$ we can write this formally as
\begin{equation}
	\ddot P(t)= ({\mathcal A}-K(t){\mathcal H})\dot P(t) + \dot P(t)({\mathcal A}-K(t)h)^{*}.
\end{equation}
The solution of this operator equation is expressed as
\begin{displaymath}
	\dot P(t)= {\mathcal S}_{k}(t,0)\dot P(0){\mathcal S}_{k}^{*}(t,0),
\end{displaymath}
where ${\mathcal S}_{k}(t,0)$ is the solution operator for the operator evolution
\begin{equation}
	\frac{dL (t)}{dt}= ({\mathcal A}-K(t)h) L(t), \hspace{.1in} L(t_{0})=L_{0}.
\end{equation}
We differentiate $K(t)$ to get another operator evolution equation:
\begin{equation}
	\frac{dK(t)}{dt}= \dot P(t)h^{*}=L(t)\dot{P}(0)L^{*}(t)h^{*}, \hspace{.1in} K(t_{0})= \dot P(0)h^{*}.
\end{equation}

Here we can obtain $\dot P(0)$ by 
\begin{equation}
	\dot P(0)= {\mathcal A}P(0)+ P(0){\mathcal A}^{*}+ P(0)h h^{*} + Q,
\end{equation}
with $Q$ replaced by $\mbox{Tr}\Lambda$ for the pure jump case. Note also that the Riccati solution is obtained by quadrature:
\begin{displaymath}
	P(t)=P(0)+\int_{0}^{t} L(r)\dot{P}(0)L^{*}(r)dr.
\end{displaymath}
For the linear feedback control problem, rigorous treatment of such Chandrasekhar system with ${\mathcal A}$ generating a $C_{0}$-semigroup has been studied in \cite{Ito1987, Ito1988} giving solvability for $L(t)$ and $K(t)$ as strongly continuous operator-valued evolutions.

\section{Acknowledgment} The first author's research has been supported by the U. S. Air Force Research Laboratory through the National Research Council Senior Research Fellowship of the National Academies of Science, Engineering and Medicine.



\begin{thebibliography}{30} 

\bibitem{Added1988} H. Added and S. Added, Equations of Langmuir turbulence and nonlinear Schr\"odinger equation: smoothness and approximation, {\it J. Functional Anal.}, {\bf 79}, (1988), 183-210.
\bibitem{Alshanskiy2014} M. A. Alshanskiy, Wiener-Itô chaos expansion of Hilbert space valued random variables, {\it Journal of Probability} {\bf 2014}, Article ID 786854, (2014).
\bibitem{Applebaum2009} D. Applebaum, {\it L\'evy Processes and Stochastic Calculus}, Cambridge University Press, Cambridge 2009.
\bibitem{Bala1980} Balakrishnan, A. V.: Nonlinear white noise theory. In {\it Multivariate analysis}, {\bf V} (Proc. Fifth Internat. Sympos., Univ. Pittsburgh), North-Holland, Amsterdam, (1980), 97-109.
\bibitem{Barbu1976} V. Barbu, {\it Nonlinear semigroups and differential equations in Banach spaces}, Springer Dordrecht, (1976).
\bibitem{Becker1979} W. Becker and H. Mitter, Quantum theory of a free electron laser, {\it Z. Physik B}, {\bf 35}, (1979), 399-404.
\bibitem{Bensoussan1971} A. Bensoussan, {\it Filtrage optimal des systemes lineaires}, Dunod, (1971).
\bibitem{Bove1974} A. Bove, G. Da Prato and G. Fano, An Existence Proof for the Hartree-Fock Time-dependent Problem with Bounded Two-Body Interaction, {\it Commun. Math. Phys.}, {\bf 37}, (1974), 183-191. 
\bibitem{Burns2015} J. A. Burns, and C. N. Rautenberg, Solutions and approximations to Riccati integral equations with values in a space of compact operators, {\it SIAM J. Control and Opt.}, Vol. 53, No. 5,(2015), 2846-2877.
\bibitem{Cazenave1998} T. Cazenave and A. Haraux, {\it An Introduction to Semilinear Evolution Equation}, Claderon Press, Oxford, (1998).
\bibitem{Chadam1972} J. M. Chadam, On the Cauchy problem for the coupled Maxwell-Dirac equations, {J. Math. Phys.}, {\bf  13}, (1972), 597-604.
\bibitem{Chadam1973} J. M. Chadam, Global solutions of the Cauchy problem for the (classical) coupled Maxwell-Dirac equations in one space dimension, {\it Journal of Functional Analysis}, {\bf 13}, (1973), 173-184.
\bibitem{Chadam1973b} J. M. Chadam, Asymptotic behavior of equations arising in quantum field theory, {\it Applicab!e Analysis}, {\bf 3}, (1973), 377-402.
\bibitem{Chadam1974} J. M. Chadam and R. T. Glassey, On certain global solutions for the Cauchy problem for (classical) coupled Klein-Gordon-Dirac equations in one and three space dimensions, {Archiv for Rational Mechanics and Analysis}, {\bf 54}, (1974), Issue.3, 223-237.
\bibitem{Chandrasekhar1960} S. Chandrasekhar, {\it Radiative Transfer}, Dover Publications, 1960.
\bibitem{Curtain1975} R. F. Curtain, Infinite dimensional estimation theory applied to a water pollution problem, In: Cea, J. (eds), {\it Optimization Techniques Modeling and Optimization in the Service of Man Part 2. Optimization Techniques} Lecture Notes in Computer Science, vol 41. Springer, Berlin, (1975), 685-699.
\bibitem{DaPrato1972} G. Da Prato, Quelques r\'esultats d'existence et r\'egularit\'e pour un probl\'eme non lin\'eaire de la th\'eorie du contr\^ole, {\it  Bull. Soc. math. France}, {\bf  31-32}, (1972), 127-132. 
\bibitem{DaPrato1973} G. Da Prato, Quelques r\'esultats d'existence unicite et regularit\'e pour un probleme de la th\'eorie du contr\^ole, {\it J. Math. Pures et Appl.}, {\bf 52}, (1973), 353-375.
\bibitem{DaPrato1996} G. Da Prato and J. Zabczyk, {\it Ergodicity for Infinite Dimensional Systems}, London Mathematical Society Lecture Note Series. 229, (1996).
\bibitem{DaPrato2014} G. Da Prato and J. Zabczyk {\it Stochastic Equations in Infinite Dimensions}, Second Edition,  Cambridge University Press, (2014).

\bibitem{Dyson1949} F. J. Dyson, The radiation theories of Tomonaga, Schwinger, and Feynman, {\it Physical Review}, {\bf 75}, No. 3, (1949), 486-502. 
\bibitem{Flato1987} M. Flato, J. Simon and E. Taflin, On global solutions of the Maxwell-Dirac equations, {\it Commun. Math. Phys.} {\bf 112}, (1987), 21-49.
\bibitem{Fernando2013} B. P. W. Fernando and S. S. Sritharan, Nonlinear Filtering of Stochastic Navier-Stokes equation with Ito-Levy noise, {\it Stochastic Analysis and Applications}, {\bf   31}, (2013), No.3, 381-426.

\bibitem{Fujisaki1972} M. Fujisaki, G. Kallianpur, G., and H. Kunita, Stochastic differential equations for the nonlinear filtering problem, {\it Osaka J. Math.}, {\bf 9} (1972), 19–40.
\bibitem{Getoor1975} R. Getoor,  On the construction of kernels. In Seminaire de probabilitiés IX. {\it Lecture Notes in Mathematics.} Meyer, P.A., ed., {\bf  465}, Springer-Verlag, Berlin (1975), 443-463.
\bibitem{Gross1960} L. Gross, Integration and nonlinear transformation in Hilbert space, {\it Trans. Amer. Math. Soc. }, {\bf  94}, (1960), 404-440.
\bibitem{Gross1962} L. Gross, Measurable functions on Hilbert space, {\it Trans. Amer. Math. Soc.}, {\bf  105}, (1960), 372-390.
\bibitem{Gross1966} L. Gross, The Cauchy problem for the coupled Maxwell and Dirac equations, {\it Communications on Pure and Applied Mathematics}, {\bf  XIX}, (1966), 1-15.
\bibitem{Gustafsson2019} J. Gustafsson, B. F. Akers, J. A. Reeger, and S. S. Sritharan, Atmospheric propagation of high energy lasers, {\it Eng. Math. Lett.}, {\bf 2019} (2019), Article ID 7.

\bibitem{Hille1996} E. Hille and R. S. Phillips, {\it Functional Analysis and Semigroups}, American Mathematical Society, Providence, RI., 1996.
\bibitem{Holden2010} H. Holden, B. Oksendal, J. Oboe, T. Zhang, {\it Stochastic Partial Differential Equations: A Modelling, White Noise Functional Approach}, Springer-Verlag, New York, 2010.
\bibitem{Hobbs1996} S. Hobbs and S. S. Sritharan, Nonlinear filtering for stochastic reaction and diffusion equations, in {\it Probability Theory and Modern Analysis},  N. Gretsky,  J. Goldstein, J., and J. J. Uhl, eds., Marcel Dekker, New York, (1996).
\bibitem{Ito1956} K. Ito, Spectral type of the shift transformation of differential processes with stationary increments, {\it Transactions in American Mathematical Society}, {\bf 81}, (1956), 253-263.
\bibitem{Ito1987} K. Ito and R. Powers, Chandrasekhar equations for infinite dimensional system, {\it SIAM Journal of Control and Optimization}, {\bf 25}, (1987), 596-611.
\bibitem{Ito1988} K. Ito and R. Powers, Chandrasekhar equations for infinite dimensional systems, II. Unbounded input and output case, {\it Journal of Differential equations}, 75, (1988), 371-402.
\bibitem{Kailath1974} T. Kailath, A view of three decades of linear filtering theory, {IEE Trans. Information Theory}, {\bf IT-20}, No. 2, (1974), 146-181.
\bibitem{Kallianpur1980} G. Kallianpur, {\it Stochastic Filtering Theory}, Springer-Verlag, New York, (1980).
\bibitem{Kallianpur1988} G. Kallianpur and R. Karandikar, {\it White Noise Theory of Prediction, Filtering and Smoothing}, Gordon and Breach Publishers, Amsterdam (1988).
\bibitem{Kalman1961} R. E. Kalman and R. S. Bucy, New results in linear filtering and prediction theory, {Journal of Basic Engineering}, March (1961), 95-108.
\bibitem{Kato1976} T. Kato, {\it Perturbation Theory of Linear Operators}, Springer-Verlag, New York, (1976).

\bibitem{Kato1976a} T. Kato, Linear and quasi-linear equations of evolutions of hyperbolic type, in: {\it Hyperbolicity}, Edited by G. Da Prato and G. Geymonat, CIME Summer School, {\bf 72},(1976),125-191, Springer-Verlag, Berlin.
\bibitem{Kunita1971}  H. Kunita, Asymptotic behavior of nonlinear filtering error of Markov processes, {\it J. of Multivariate Analysis}, {\bf 1}, (1971), 365-393.
\bibitem{Kunita1991}  H. Kunita, Ergodic properties of nonlinear filtering processes, in K. S. Alexander and J. Watkins editors {\it  Spatial stochastic processes : a festschrift in honor of Ted Harris on his seventieth birthday }, Springer-Verlag (1991), 234-259.
\bibitem{Kuiper1980} H. J. Kupier and S. M. Shew, Strong solutions for infinite dimensional Riccati equations arising in transport theory, {\it SIAM J. Math. Anal.}, {\bf 11}, (1980), No.2, 211-222.
\bibitem{Kuiper1985} H. J. Kupier, Generalized operator Riccati equations, {SIAM J. Math. Anal.}, {\bf 16}, (1985), No. 4, 675-694.
\bibitem{Kuo1975} H-H. Kuo, {\it Gaussian Measures in Banach Spaces}, Springer-Verlag, New York, (1975).
\bibitem{Lions1971} J. L. Lions, {\it Optimal Control of Systems Governed by Partial Differential Equations}, Springer-Verlag, Berlin, 1971.
\bibitem{Mazz1988} G. Mazziotto, L. Stettner, J. Szpirglas, and J. Zabczyk, On impulse control with partial observation, {SIAM J. of Control and Optimization}, {\bf  26}, (1988), No.4, 964-984.
\bibitem{Madey2010} J. M. Madey, Invention of the Free Electron Laser, {\it Reviews of Accelerator Science and Technology, } {\bf  3}, (2010) 1–12.
\bibitem{Meyer1973} G. H. Meyer, Edited, {Initial Value Methods for Boundary Value Problems: Theory and Applications of Invariant Embedding}, Academic Press, New York, 1973.
\bibitem{Mohan2017} M. T. Mohan and S. S. Sritharan, $L^{p}$-solutions to stochastic Navier-Stokes equation with L\'evy noise with $L^{m}$ initial data, {\it Evolution Equation and Control Theory}, {\bf 6},(2017), No. 3, 409-425.
\bibitem{Mohan2018} M. T. Mohan and S. S. Sritharan, Stochastic quasilinear symmetric hyperbolic system perturbed by L\'evy noise, {\it Journal of Pure and Applied Functional Analysis}, {\bf 3}, (2018), No.1,  137-178.
\bibitem{Papa1973} G. C. Papanicolaou, D. McLaughlin, and R. Burridge, A Stochastic Gaussian Beam, {\it J. Math, Phys.}, {\bf 14}, (1973), No.1, 84-89.

\bibitem{Pazy1983} A. Pazy, {\it Semigroups of Linear Operators and Applications to Partial Differential Equations}, Springer-Verlag, New York, 1983.
\bibitem{Reed1975a} M. C. Reed, Higher order estimates and smoothness of nonlinear wave equations, {\it Proceedings of the American Mathematical Society}, {\bf  51}, (1975),No.1, 79-85.
\bibitem{Reed1975} M. C. Reed and B. Simon, {\it Methods of Modern Mathematical Physics,Vol. II: Fourier Analysis and Self-adjointness}, Academic Press, 1975.
\bibitem{Reed1976} M. C. Reed, {\it Abstract Nonlinear Wave Equations}, Lecture Notes in Mathematics, {\bf 507}, Springer-Verlag, Berlin 1976.
\bibitem{Reed1980} M. C. Reed and B. Simon, {\it Methods of Modern Mathematical Physics,Vol. I: Functional Analysis}, Academic Press, 1975.
\bibitem{Sato1969}  H. Sato, Gaussian measure on a Banach space and abstract Wiener measure, {\it Nagoya Math. J.} {\bf  36}, (1969), 65-81.
\bibitem{Segal1956} I. E. Segal, Tensor algebras over Hilbert spaces, {Trans. Amer. Math. Soc.} {\bf  81}, (1956), 106-134.
\bibitem{Segal1963} I. E. Segal, Nonlinear Semigroups, {\it Annals of Mathematics}, {\bf 78}, (1963),No.2, 339-363.
\bibitem{Simon1974} B. Simon, {\it The $P(\Phi)_{2}$ Euclidean (Quantum) Field Theory}, Princeton Series in Physics, Princeton, NJ, 1974.
\bibitem{Sp2002} P. Sprangle, J. R. Peñano, and B. Hafizi, 	Propagation of intense short laser pulses in the atmosphere, {\it Phys. Rev.} {E \bf 66}, Issue: 4, 2002, 1-21.
\bibitem{Sprangle2003} P. Sprangle, J. P. Pe\~nano, A. Ting, B. Hafizi and D. E. Gordon, Propagation of short, High-intensity Laser pulses in air, {\it Journal of Directed Energy}, {\bf 1}, (2003), 73-92.
\bibitem{Sritharan1995}  S. S. Sritharan, Nonlinear filtering of stochastic Navier-Stokes equations. In {\it Nonlinear Stochastic PDEs: Burgers Turbulence and Hydrodynamic Limit},T. Funaki and W. A. Woyczynski, eds., Springer-Verlag, New York, (1995) 247–260.
\bibitem{Sritharan2010} S. S. Sritharan and M. Xu, Convergence of particle filtering method for nonlinear estimation of vortex dynamics. {\it Comm. Stoch. Anal.} {\bf  4}, (2010), No. 3, 443–465.
\bibitem{Sritharan2011} S. S. Sritharan and M. Xu, A stochastic Lagrangian particle model and nonlinear filtering for three dimensional Euler flow with jumps, {\it Comm. Stoch. Anal.}, {\bf  5}, (2011), No.3, 565–583.
\bibitem{Sritharan2023} S. S. Sritharan and S. Mudaliar, Stochastic quantization of Laser propagation models, {\it Infinite Dimensional Analysis, Quantum Probability and Related Topics}, Published on-line July 2023.
\bibitem{Sritharan2023a} S. S. Sritharan and S. Mudaliar,	“Nonlinear filtering of classical and quantum spin systems”, Journal of Stochastic Analysis, Vol. 4, (2023), No. 4, pp 1-25.
\bibitem{Strato1960} R. L. Stratonovich, Conditional Markov processes, {\it Theory of Probability and its Applications}, {\bf V}, (1960), No.2, 156-178.
\bibitem{Strato1960} R. L. Stratonovich, Conditional Markov processes, {\it Theory of Probability and its Applications}, {\bf V}, (1960), No.2, 156-178.
\bibitem{Strauss1978} W. Strauss, Nonlinear invariant wave equations, {\it Invariant Wave Equations}, Lecture notes in Physics, No. 78,  (1978), 197-249, Springer-Verlag, New York.
\bibitem{Strauss1989} W. Strauss, {\it Nonlinear Wave Equations}, CBMS Regional Conference Series No.73, American Mathematical Society, Providence,1989.
\bibitem{Strohbehn1978} J. W. Strohbehn, {\it Laser Beam Propagation in the Atmosphere}, Springer-Verlag, New York, 1978.
\bibitem{Sulem1999} C. Sulem and P-L. Sulem, {\it Nonlinear Schr\"odinger Equations: Self-Focusing and Wave Collapse}, Springer-Verlag, New York 1999.
\bibitem{Szpirglas1978} J. Szpirglas, Sur l’équivalence d’équations différentielles stochastiques à valeurs mesures intervenant dans le filtrage markovien non linéaire, {\it Ann. Inst. Henri Poincaré,} 	{\bf  XIV}, (1978), No.1, 33-59. 
\bibitem{Tap1977} F. D. Tappert, The parabolic approximation method, in  Joseph B. Keller, John S. Papadakis, Editors, {\it Wave Propagation and Underwater Acoustics}, Springer-Verlag, Berlin, 1977.
\bibitem{Tartar1974} L. Tartar, Sur I’\'Etude Directe d'Equations non Lin\'eaires lntervenant en Th\'eorie du Contr\^ole Optimal, {J. Functional Analysis}, {\bf 6}, (1974), 1-47.
\bibitem{Temam1971} R. Temam, Sur I'\'equation de Riccati associ\'ee \'a des op\'erateurs non born\'es, en dimension infinie,{J. Functional Analysis}, {\bf 7}, (1971), 85-115.
\bibitem{vanNeerven2020} J. van Neerven, M. Veraar, Maximal inequalities for stochastic convolutions in 2-smooth Banach spaces and applications to stochastic evolution equations, {\it Phil. Trans. R. Soc. A} {\bf 378}: 20190622.
\bibitem{Tatarski1961} V. I. Tatarski, {\it Wave Propagation in a Turbulent Medium}, Dover Publishers, New York, 1961.
\bibitem{Wiener1949} N. Wiener, {\it Extrapolation, Interpolation and Smoothing of Stationary Time Series, with Engineering Applications}, Press and Wiley, (1949).
\bibitem{Yosida1991} K. Yosida, {\it Functional Analysis}, Springer-Verlag, New York, (1991). 
\bibitem{Zakai1969}  M. Zakai, On the optimal filtering of diffusion processes, {\it Wahrsh, Z. Verw. Gebiete}, {\bf 11}, (1969), 230–243.
\bibitem{Zhu2017} J. Zhu, Z. Brzezniak and E. Hausenblas, Maximal inequalities for stochastic convolutions driven by compensated Poisson random measures in Banach spaces, {\it Annales de l’Institut Henri Poincar\'e},  (2017), Vol. 53, No. 2, 937–956.
\bibitem{Zohuri2016} B. Zohuri, {\it Directed Energy Weapons: Physics of High Energy Lasers (HEL)}, Springer-verlag, New York, 2016.













\end{thebibliography}
\end{document}